\documentclass{amsart}

\usepackage{mabliautoref}
\newenvironment{itemise}{\begin{itemize}}{\end{itemize}}

\usepackage{amsthm}
\usepackage{amstext}
\usepackage{enumitem}
\usepackage{amssymb}
\usepackage{amsmath,calligra,mathrsfs}
\usepackage[all,cmtip]{xy}
\usepackage{dsfont}
\usepackage{hyperref}
\usepackage{graphicx}
\usepackage{tikz-cd}
\newcommand\cal{\mathcal}
\newcommand\bb{\mathbb}

\usepackage{xcolor}
\usepackage{euscript}

\theoremstyle{plain}

\theoremstyle{definition}

\theoremstyle{remark}
\newtheorem*{rem}{Remark}

\theoremstyle{plain}

\DeclareMathOperator{\Hom}{Hom}

\DeclareMathOperator{\length}{length}

\DeclareMathOperator{\rank}{rank}
\DeclareMathOperator{\Spec}{Spec}
\DeclareMathOperator{\Proj}{Proj}

\DeclareMathOperator{\sheafhom}{\mathscr{H}\text{\kern -3pt {\calligra\large om}}\,}
\DeclareMathOperator{\sheafder}{\mathscr{D}\text{\kern -3pt {\calligra\large er}}\,}

\DeclareMathOperator{\codim}{codim}
\DeclareMathOperator{\perf}{perf}

\DeclareMathOperator{\red}{red}

\DeclareMathOperator{\Der}{Der}

\def\bP{\mathbb{P}}

\def\Ker{\text{Ker}}

\def\Im{\text{Im}}

\let\phi\varphi
\usepackage{comment}

\author{Lena Ji}
\address{Department of Mathematics\\
Princeton University\\
	Fine Hall\\
	Washington Road \\
	Princeton, NJ 08544 \\
	USA}
\email{lji@math.princeton.edu}
\author{Joe Waldron}
\address{Department of Mathematics\\
Michigan State University\\
619 Red Cedar Road\\
East Lansing, MI 48824 \\
USA}
\email{waldro51@msu.edu}

\begin{document}

\title[Geometrically non-reduced varieties]{Structure of geometrically non-reduced varieties}

\begin{abstract}
	We prove a structural result for geometrically non-reduced varieties and give applications to Fano varieties. For example, we show that if $X$ is the generic fibre of a Mori fibre space of relative dimension $n$, and the characteristic is $p>2n+1$, then any geometric non-reducedness of $X$ comes from the base of some fibration.
\end{abstract}

\maketitle

\tableofcontents

\section{Introduction}

Given a fibration $f\colon \EuScript{X}\to S$ of smooth varieties over an algebraically closed field $k$, the generic fibre $X$ is a regular variety over the function field $K=K(S)$. In characteristic 0 this implies that $X$ is smooth over $K$.  On the other hand, a key difficulty in positive characteristic comes from the existence of regular varieties $X$ over $K$ for which $X\otimes_{K}\overline{K}$ is singular.  This corresponds to a fibration  $f$ with every fibre singular; the most familiar examples of such objects are the quasi-elliptic fibrations which occur in the Bombieri--Mumford classification of smooth surfaces in characteristics $2$ and $3$.  Furthermore, Example \ref{exmp:pm_fermat_example} gives regular Fano varieties in characteristic $p<\dim X+2$ for which $X\otimes_{K}\overline{K}$ is non-reduced.  Such a variety is \emph{geometrically non-reduced}.

These fibrations cause difficulties when applying standard methods of characteristic zero birational geometry, for example fibre space adjunction \cite{ejiri_positivity_2016} and semi-positivity  \cite{patakfalvi_semi-positivity_2014-1}.  Therefore a key part of extending birational geometry to higher dimensions in positive characteristic will be to understand and control the occurrence of these fibrations.

One way to do this is to compare a geometrically non-reduced variety $X$ over $K$ to the variety $Y$ over $\overline{K}$ which one obtains by taking the normalisation of the maximal reduced subscheme of $X\otimes_K \overline{K}$ with natural morphism $\phi\colon Y\to X$.   

The earliest result of this kind was the genus change formula of Tate \cite{tate}, which implies that if $X$ is a regular curve and $X\otimes_K\overline{K}$ is reduced, there is  an effective Weil divisor $C$ such that
\[K_Y+(p-1)C\sim \phi^*K_X.
\]

This formula was recently extended to arbitrary dimensional normal varieties, including the geometrically non-reduced case, by Patakfalvi and the second author \cite{pw}.   This enabled them to prove for example that regular del Pezzo surfaces over a field of characteristic $p\geq 11$ are smooth, and hence techniques such as those of \cite{patakfalvi_semi-positivity_2014-1} can be applied to del Pezzo fibrations of relative dimension $2$ in these characteristics.   Further applications of this formula to del Pezzo surfaces have appeared in \cite{tanaka2019}, \cite{bernasconi_tanaka}, and \cite{tanaka_boundedness}.  We discuss applications of the present article to Fano varieties (including higher dimensional ones)  in \autoref{sub:introduction_Fano}.

\subsection{Canonical bundle formula}

One of the main aims of the present article is to reveal the geometric significance of the Weil divisor $C$ which appeared in the formula above. That is, we determine a canonical linear system $\frak C\subseteq |C|$ which is intricately connected with the non-normal singularities of $X\otimes_K K^{1/p}$.

\begin{theorem}\label{main_theorem}
	Let $K$ be the function field of a variety over a perfect field of characteristic $p>0$. Let $X$ be a projective normal variety over $K$, and suppose that $H^0(X,\cal O_X)=K$.  Let $L$ be a field lying between $K$ and $K^{{1}/{p}}$, 
	and let $Y$ be the normalisation of $(X\otimes_K L)_{\red}$ with morphism $\phi\colon Y\to X$.
	
	Then there is a canonically determined linear system $\frak{C}$ of Weil divisors with fixed part $\frak{F}$ and movable part $\frak{M}$ such that
	\[K_Y+(p-1)\frak{C}\sim \phi^*K_X.
	\]

Furthermore, $\frak{F}$ and $\frak{M}$ satisfy the following properties:

\begin{enumerate}
	\item\label{F_normality} $\frak{F}=0$ if and only if $(X\otimes_K L)_{\red}$ is regular in codimension $1$.
	\item\label{F_support} $\frak{F}$ has support equal to that of the conductor divisor of $Y\to (X\otimes_KL)_{\red}$.
	\item\label{M_reducedness} $\frak M=0$ if and only if $X\otimes_K L$ is reduced.
\end{enumerate}
	Let $f\colon X\dashrightarrow V$ be the Stein factorisation of the rational map induced by $\frak{M}$ and let $W$ be the normalisation of $(V\otimes_KL)_{\red}$.  
	\begin{enumerate}
	\setcounter{enumi}{3}
	\item\label{fibre_reduced}If $X_{\xi}$ is the generic fibre of $f$ then $X_\xi\otimes_{K(V)} K(W)$ is reduced.
	
	\item\label{maximal_reduced}  If $g\colon X\dashrightarrow Z$ is a rational map with generic fibre $X_{\eta}$ such that $X_\eta\otimes_{K(Z)} K((Z\otimes_K L)_{\red})$ is reduced, then there is a factorisation $X\dashrightarrow Z\dashrightarrow V$.

	\item\label{pullback_from_w} If $\frak{M}_W$
	is the linear system on $W$ associated to the base change $V\otimes_K L$, then there is a natural map $g\colon Y\dashrightarrow W$ and $\frak{M}=g^*\frak{M}_W$.
	\end{enumerate}
	
\end{theorem}

In summary,  the first two points tell us that the fixed part $\frak{F}$ measures the failure of normality of $(X\otimes_K L)_{\red}$, while the next three tell us that the movable part $\frak{M}$ measures the failure of reducedness of $X\otimes_K L$.  In particular, the last two points give a universal property satisfied by $f$:  it is the deepest of all morphisms whose generic fibres stay reduced under the base change $L/K$.

It should be noted that $\frak C$ may not be, and often is not, equal to the complete linear system $|C|$.  However, we can use it to read some consequences from knowledge of $|C|$.
In the following corollaries, let $X$ be a normal variety over an arbitrary field $K$ with $H^0(X,\cal O_X)=K$, and $f\colon Y\to X$ be the normalised base change by $L/K$ for an arbitrary field extension $L$.

\begin{corollary}
	If $C$ is fixed, then $X\otimes_K L$ is reduced.
	\end{corollary}

\begin{corollary}\label{C_empty}
	$|C|=\emptyset$ if and only if $X\otimes_K L$ is normal.
	\end{corollary}

Note that \autoref{C_empty} appeared recently in \cite{tanaka2019}, as a consequence of \cite{pw}, but we mention it here for completeness.

\subsection{Essential part of a base change}

On the way to \autoref{main_theorem} we construct a particular subextension $L/L'/K$ which we call the \emph{essential part} of the base change.  This is the universally smallest subextension which reveals all the non-reducedness of $X\otimes_K L$, in the following precise sense:

\begin{theorem}[\autoref{trick}]\label{thm:trick}
	
	Let $K$ be the function field of a variety over a perfect field of characteristic $p>0$. Let $X$ be a projective normal variety over $K$, and suppose that $H^0(X,\cal O_X)=K$. Let $L$ be a field extension such that $K^{1/p}\supseteq L\supseteq K$.

	There exists a unique subextension $L/L'/K$ such that the varieties
	\[	\xymatrix{
		Z=(X\otimes_K L)_{\red}\ar[r]\ar[d] & Z'=(X\otimes_K L')_{\red}\ar[r]\ar[d] & X\ar[d]\\
		\Spec(L)\ar[r] & \Spec(L')\ar[r] &\Spec(K)	
	}
	\]
	satisfy:
	\begin{enumerate}
		
		\item\label{itm:same_deg}
			All of the following properties, which are equivalent for any such diagram:
			\begin{enumerate}
				\item 
		$Z\to Z'$ and $\Spec(L)\to \Spec(L')$ are morphisms of the same degree.
		\item   $Z'\otimes_{L'}L$ is reduced.
		\item  $Z'\otimes_{L'} L\cong Z$.
		\end{enumerate}
		\item \label{itm:subext_factor} If $L/L''/K$ is a subextension satisfying the properties \autoref{itm:same_deg}, then $L'\subseteq L''$.
		
			\item\label{itm:omega_zero}
		$H^0(Z', \Omega_{Z'/X}^\vee )=0$.
		
	\end{enumerate}
Furthermore, $L'$ is characterised as the  unique subextension satisfying both \autoref{itm:same_deg} and \autoref{itm:subext_factor}, and also as the unique subextension satisfying both \autoref{itm:same_deg} and \autoref{itm:omega_zero}
\end{theorem}

Note that it is possible to define the essential part for a field extension of arbitrary height by iteratively applying the above construction.

\subsection{Fano varieties}\label{sub:introduction_Fano}

Our investigation of the structure of geometrically non-reduced varieties is ultimately motivated by an ongoing investigation of the structure of geometric singularities of Fano varieties over imperfect fields begun in \cite{pw}.  Over a fixed algebraically closed field of characteristic zero, it is known that regular Fano varieties of fixed dimension form a bounded family by \cite{kollar_rational_1992}.  Furthermore, it is known that if one allows certain controlled singularities  (so called $\varepsilon$-lc for fixed $\varepsilon$) then the same remains true by a theorem of Birkar \cite{birkar_singularities_2016} (once known as the Borisov--Alexeev--Borisov conjecture).   It is an interesting question whether these statements extend to positive characteristic, and to varieties over non-closed fields.  A consequence of this would be that regular Fano varieties of dimension $n$ should exhibit pathologies of positive characteristic such as geometric non-reducedness only in small characteristics depending on $n$.   Conversely, ruling out such pathologies directly can be a useful step towards proving boundedness results.   Of particular interest are those Fano varieties over non-algebraically closed fields which satisfy $\rho=1$, for these occur as generic fibres of Mori fibre spaces over algebraically closed fields.

The following theorem is a result in this direction.   
It implies that if $X$ is a generic fibre of a Mori fibre space, in a characteristic which is large compared to $\dim(X)$, $X$ can only fail to be geometrically reduced due to the base of some fibration.

\begin{theorem}\label{thm:MFS}
	Let $K$ be the function field of a variety over a perfect field of characteristic $p>0$. Let $X$ be a variety over $K$, and assume the following:
	\begin{itemise}
		\item $X$ is normal, $\mathbb{Q}$-factorial and $\rho(X)=1$,
		\item $X$ is geometrically irreducible over $K$,
		\item there is an effective divisor $B$ such that $K_X+B\equiv 0$, and
		\item if $n=\dim(X)$ then $p>2n+1$.
	\end{itemise}
	
	Then there is a birational morphism $\phi\colon\tilde{X}\to X$ and a contraction $\widetilde{X}\to V$ of relative dimension at least $1$, such that $(\widetilde{X}\otimes_K{K^{1/p}})_{\red}$ is birational to $\widetilde{X}\times_V(V\otimes_K{K}^{1/p})_{\red}$.
\end{theorem}

Note that there are no assumptions on the coefficients of $B$.  In characteristic zero, if one assumes additionally that $X$ is of Fano type, one would be tempted to prove that $V$ is also of Fano type and proceed by induction on dimension using the LMMP.  However, note that this approach is more difficult in positive characteristic, see the example in \cite[Theorem 1.4]{tanaka_pathologies_2016}.

We can also apply our results to give restrictions on the existence of geometrically non-reduced del Pezzo surfaces in small characteristics, in the style of \cite[Theorem 4.1]{pw}.  In the following corollaries, $X$ will be a regular del Pezzo surface over a field $K=H^0(X,\cal O_X)$ of characteristic $p>0$, which is geometrically non-reduced over $K$.  $Y$ will be the normalisation of $(X\otimes_K \overline{K})_{\red}$.  In these results  we refer to the classification of the possibilities for $(Y,(p-1)C)$ obtained in \cite[Theorem 4.1]{pw}, and the notation used there.

\begin{corollary}\label{cor:dp_hirzebruch}
	$(Y,(p-1)C)$ cannot be of the type $(H_d,D)$.
\end{corollary}

\begin{corollary}\label{regular_dp}
	If $p=3$  then $(X\otimes_K\overline{K})_{\red}$ is normal.
\end{corollary}

\begin{corollary}\label{cor:dp_product}
	Suppose that $(Y,(p-1)C)$ is of the type $(\bP^1\times\bP^1,F)$.  Then there exists a contraction $X\to V$ to a geometrically non-reduced curve $V$ such that over the generic point of $V$
	\[Y=X\times_V ((V\otimes_K \overline{K})_{\red}^\nu).
	\]
	\end{corollary}

Note that these consequences  have been chosen for illustration, and the list is far from exhaustive.   For example it is possible to combine the ideas of the above corollaries to give more precise descriptions of what happens when $X$ is geometrically non-reduced and $(X\otimes_K \overline{K})_{\red}$ is non-normal.  We leave this to the interested reader.

\subsection{Structure of the paper}
We first review background material in \autoref{sec:background}. Then in \autoref{sec:reinterpret} we recall the setting and some of the results of \cite{pw} and begin to describe the linear system $\frak C$. In \autoref{sec:universal_property} we study the fixed part $\frak F$ and prove \autoref{main_theorem} \autoref{F_normality} and \autoref{F_support}. Then in \autoref{sec:section_M} we define the movable part $\frak M$ and prove \autoref{thm:trick}
and \autoref{main_theorem} \autoref{M_reducedness}.
In \autoref{sec:morphism} we describe the map $X\dashrightarrow V$ and prove \autoref{main_theorem} \autoref{fibre_reduced}, \autoref{maximal_reduced}, and \autoref{pullback_from_w}.
Finally, in \autoref{sec:fano_section} we give geometric applications and prove \autoref{thm:MFS}, and we give some concrete examples in \autoref{sec:examples}.

In summary, the parts of \autoref{main_theorem} are proven in the following locations:
\begin{enumerate}
\item[\autoref{F_normality}] \autoref{cor:almost_conductor_normalisation},
\item[\autoref{F_support}] \autoref{cor:almost_conductor_normalisation},
\item[\autoref{M_reducedness}] \autoref{prop:M=0_iff_reduced},
\item[\autoref{fibre_reduced}] \autoref{prop:base_change_w_reduced},
\item[\autoref{maximal_reduced}] \autoref{new_univ_prop},
\item[\autoref{pullback_from_w}] \autoref{new_univ_prop}.
\end{enumerate}

\subsection*{Acknowledgments}

The authors would like to thank Fabio Bernasconi, J\'anos Koll\'ar, James M\textsuperscript{c}Kernan and Zsolt Patakfalvi for helpful conversations, and Johan de Jong for useful comments and suggesting a proof of \autoref{lem:Omega_of_normalisation}. They also thank the referee for a very careful reading.
In particular, the first author is grateful to her advisor J\'anos Koll\'ar for his constant support, and to David Eisenbud and UC Berkeley for hosting her in the Spring 2019 semester.

This material is based upon work of the first author while supported by the National Science Foundation Graduate Research Fellowship Program under Grant No. DGE-1656466; and of the second author while supported by the National Science Foundation under
Grant No. 1440140, while he was in residence at the Mathematical Sciences Research Institute in Berkeley, California, during the Spring of 2019.
 Any opinions, findings, and conclusions or recommendations expressed in this material are those of the authors and do not necessarily reflect the views of the National Science Foundation.

\section{Preliminaries}\label{sec:background}

In this paper, a variety over a field $K$ will be an integral separated scheme of finite type over $K$.
Unless stated otherwise, we may work over imperfect fields, and all varieties will be quasi-projective.

\subsection{Fields and geometric singularities}\label{arbitrary_fields}

\begin{notation}\label{not:fields}
Let $K$ be a field.  
\begin{enumerate}
	\item We consider a fixed algebraic closure $\overline{K}$ of $K$.
	\item By $K^{1/p^n}$ we mean the target of the composition of $n$ Frobenius homomorphisms, that is $F^n\colon K\hookrightarrow K^{1/p^n}$.  This appears naturally as a subring of $\overline{K}$, in which we have adjoined the $(p^{n})^{\mathrm{th}}$ root of every element of $K$.
	\item  $K^{1/p^\infty}=\cup_{n}K^{1/p^n}$ denotes the perfect closure of $K$ inside $\overline{K}$.
	\item If $L/K$ is a finite purely inseparable extension, the minimal $n$ such that $K\subseteq L\subseteq K^{1/p^n}$ is said of be the \emph{height} of $L/K$.
	\end{enumerate}
\end{notation}

If $X$ is a variety over a field $K$, and $L$ is a field extension, denote the normalisation of the maximal reduced subscheme of $X\otimes_KL$ by $(X\otimes_K L)_{\red}^{\nu}$.   To obtain geometric consequences, we are usually interested in the singularities of $(X\otimes_K K^{1/p^\infty})_{\red}^\nu$.
The assumptions on the field $K$ and the  extension $L/K$ in \autoref{main_theorem} may at first glance appear restrictive.  However, this is not the case, as general field extensions can be reduced to this case using the following lemmas.  Firstly, a single Frobenius base change is sufficient to reveal geometric singularities of our variety:

\begin{lemma}\cite[\href{https://stacks.math.columbia.edu/tag/035X}{Tag 035X}, \href{https://stacks.math.columbia.edu/tag/038O}{Tag 038O}, \href{https://stacks.math.columbia.edu/tag/038V}{Tag 038V}]{stacks},\cite[Proposition 2.10]{tanaka2019}\label{lem:geometric_single_base_change}
	Let $X$ be a scheme of finite type over a field $K$.  The following are equivalent:
	\begin{itemise}
		\item $X$ is geometrically reduced (resp. geometrically normal, smooth) over $K$.
		\item $X\otimes_K K^{1/p}$ is reduced (resp. normal, regular).
		\item $X\otimes_K K^{1/p^\infty}$ is reduced (resp. normal, regular).
		\item $X\otimes_K \overline{K}$ is reduced (resp. normal, regular).
	\end{itemise}
\end{lemma}

Next, we may usually reduce questions about varieties over general fields to those over function fields  using the following:

\begin{lemma}\label{lem:spreading_out}
	 Let $X$ be a variety over a field $k$.  Then there is a field $K\subset k$ that is the function field of a variety over a perfect field and a variety $X_K$ over $K$ such that $X\cong X_K\otimes_K k$.
	 
	 Furthermore, we also have 
	 \[(X_K\otimes_K K^{1/p^{\infty}})^{\nu}_{\red}\otimes_{K^{1/p^{\infty}}}k^{1/p^{\infty}}\cong (X\otimes_k k^{1/p^{\infty}})_{\red}^\nu.
	 \]
	 \end{lemma}
\begin{proof}
	Let $X_i=\Spec(A_i)$ where $A_i=k[x_1^i,...,x_{n_i}^i]/I_i$ is a finite set of affine charts of $X$, $X_{ij}=X_i\cap X_j\subset X_i$, and $\phi_{ij}\colon X_{ij}\to X_{ji}$ are the gluing maps.   Furthermore, for each $i,j$, fix  an affine cover of $X_{ij}$ of the form $\{U_{ijk}=\Spec(A_i[\tfrac{1}{f_{ijk}}])\}$.
	
	 Let $\{a_m\}\subset k$ be a finite set consisting of 
	 \begin{itemize}
	 	\item
	 the $k$-coefficients of a finite set of generators of $I(X_i)\subset k[x_1^i,...,x_{n_i}^i]$ for all $i$,
\item 
the $k$-coefficients of a finite set of generators of $I(X_i\setminus X_{ij})\subset A_i$ for all $i$ and $j$,
\item 
the $k$-coefficients of a polynomial representative of $f_{ijk}\in A_i$ for all $i$, $j$ and $k$, and
\item 
the $k$-coefficients of a set of polynomial representatives of a set of functions defining $\phi_{ij}|_{U_{ijk}}$, for all $i$, $j$ and $k$.
 \end{itemize}
	Let $B$ be the normalisation of the finitely generated $\mathbb{F}_p$-algebra $\mathbb{F}_p[\{a_m\}]\subset k$.  Then we can define varieties $\EuScript{X}_i\subset \mathbb{A}_B^{n_i}$, open subsets $\EuScript{X}_{ij}\subset \EuScript{X}_i$, affine covers $\{\EuScript{U}_{ijk}\}$ of $\EuScript{X}_{ij}$, and morphisms $\psi_{ij}\colon \EuScript{X}_{ij}\to \EuScript{X}_{ji}$ using the same equations. 
	If we tensor these up to $K=\mathrm{Frac}(B)$, they then glue to form the required variety $X_K$.  The necessary identities of the gluing maps hold over $K$ because they hold after applying $\otimes_K k$, and field extensions are faithfully flat.
	
	For the furthermore, we know that $(X_K\otimes_K K^{1/p^{\infty}})^{\nu}_{\red}$ is a normal variety over a perfect field $K^{1/p^\infty}$, and so is geometrically normal over $K^{1/p^{\infty}}$ by \autoref{lem:geometric_single_base_change}.  Thus
	$(X_K\otimes_K K^{1/p^{\infty}})^{\nu}_{\red}\otimes_{K^{1/p^{\infty}}}k^{1/p^{\infty}}
	$
is already normal.
\end{proof}

Finally, the singularities stop becoming worse  after performing finitely many Frobenius base changes:

\begin{lemma}\label{lem:finite_subext}
	Let $X$ be a variety over a field $K$.  There is a finite purely inseparable extension $L/K$ such that 
\[(X\otimes_K L)_{\red}^{\nu}\otimes_{L}K^{1/p^{\infty}}\cong (X\otimes_K K^{1/p^{\infty}})_{\red}^\nu.\]
In particular, we could take $L=K^{1/p^n}$ for some $n$.
	\end{lemma}
\begin{proof}
	First, as $X$ is of finite type, there exists a finite affine cover, and so by taking the compositum of a set of field extensions satisfying the condition for each affine chart we may assume that $X$ is affine.  Note that each purely inseparable field extension $L\supset K$ embeds uniquely into $K^{1/p^\infty}$.   So we assume that $X$ is affine.  Then $(X\otimes_K K^{1/p^\infty})_{\red}^{\nu}$ is also affine, and so take an embedding $Z:=(X\otimes_K K^{1/p^\infty})_{\red}^{\nu}\subset\mathbb{A}^n_{K^{1/p^\infty}}$.  Let $L$ be the finite extension of $K$ obtained by adjoining all coefficients of equations of a generating set of the ideal of $Z$.  Then there is $Z_L\subset \mathbb{A}_L^n$ defined by the same equations such that $Z_{L}\otimes_L K^{1/p^\infty}\cong (X\otimes_K K^{1/p^\infty})^\nu_{\red}$, and hence $Z_L$ is geometrically normal by Lemma 2.2. 
	
	After further extending $L$ we may assume that the morphism $Z\to X\otimes_K K^{1/p^\infty}$ descends to $L$, that is we obtain $Z_L\to X\otimes_K L$, and again after further extending $L$ we may assume that there is an open subset such that the inverse of $Z\to (X\otimes_K K^{1/p^\infty})_{\mathrm{red}}$ descends, i.e. so that $Z_L\to (X\otimes_K L)_{\mathrm{red}}$ is birational.  We also know that this is finite, so then as $Z_L$ is normal, $Z_L\cong (X\otimes_K L)_{\mathrm{red}}^\nu$, and so we are done.
	
\end{proof}

Given the finite extension $L/K$ in \autoref{lem:finite_subext}, we can canonically decompose it into a sequence of height one extensions $L_i=L\cap K^{1/p^i}$, so that $L_i/K$ has height $i$.  Our results can be applied to each of the extensions $L_i/L_{i-1}$ sequentially, or to a further decomposition of this sequence.

\begin{remark}
	It is likely that most of our arguments can be applied directly with the assumption on the ground field relaxed to \emph{differentially finite over a perfect subfield}.  However, due to the above reduction to the general case and the fact that the function field situation is the one which arises naturally in practical applications, we have not verified this.
	\end{remark}

\subsection{Properties of sheaves and purely inseparable Galois theory}

\subsubsection{Saturated subsheaves}
\begin{definition}
Let $X$ be a scheme,
and let $\cal G\subset\cal F$ be coherent sheaves on $X$.
$\cal G$ is \emph{saturated} in $\cal F$ if the quotient $\cal F/\cal G$ is torsion free,
or equivalently if
a section of $\cal F$ is in $\cal G$ if and only if its image in $\cal F_{\xi}$ is in $\cal G_{\xi}$ where $\xi$ is the generic point of $X$.  For any inclusion of sheaves $\cal G\subset\cal F$, there is a unique \emph{saturation} of $\cal G$ in $\cal F$.
\end{definition}

\begin{lemma}\label{lem:pullback_saturated_saturated}
	Let $f\colon X\to Y$ be a flat morphism of integral schemes, and let $\cal G\subset\cal F$ be coherent sheaves on $Y$.
	\begin{enumerate}
	\item 
	If $\cal G$ is saturated in $\cal F$, then $f^*\cal G$ is saturated in $f^*\cal F$.
	\item 
	$f^*\cal G^{\mathrm{sat}}=(f^*\cal G)^{\mathrm{sat}}$.
	\end{enumerate}
\end{lemma}
\begin{proof}
	We may assume that $X=\Spec(A)$, $Y=\Spec(B)$, $\cal G=\tilde{M}$ and $\cal F=\tilde{N}$.
	\begin{enumerate}
		\item
		If $M$ is saturated in $N$, then $N/M$ is torsion free.
		By \cite[\href{https://stacks.math.columbia.edu/tag/0AXM}{Tag 0AXM}]{stacks}
		$(N/M)\otimes_A B\cong(N\otimes_A B)/(M\otimes_A B)$ is also torsion free.
		\item
		Since $B$ is flat over $A$ we have an inclusion $M\otimes_A B\hookrightarrow M^{\mathrm{sat}}\otimes_A B$ and so $(M\otimes_A B)^\mathrm{sat}\hookrightarrow (M^{\mathrm{sat}}\otimes_A B)^\mathrm{sat}=M^{\mathrm{sat}}\otimes_A B$.
		If $\sum n_i\otimes b_i\in M^{\mathrm{sat}}\otimes_A B$, then for each $i$
		there is some $a_i\in A$ with $a_in_i\in M$. So letting $a=\prod a_i$
		we have $a\sum n_i\otimes b_i=\sum an_i\otimes b_i\in M\otimes_A B$.
	\end{enumerate}
\end{proof}

\subsubsection{Reflexive sheaves}

We record here some properties of reflexive sheaves; see \cite{hartshornerefl} for more details.
We denote the dual of a coherent sheaf $\cal F$ on $X$ by $\cal F^\vee:=\sheafhom_{\cal O_X}(\cal F,\cal O_X)$.

\begin{definition}
A coherent sheaf $\cal F$ on a Noetherian integral scheme $X$ is said to be \emph{reflexive} if the natural morphism
\[\cal F\to\cal F^{\vee\vee}:=\sheafhom_{\cal O_X}(\sheafhom_{\cal O_X}(\cal F,\cal O_X),\cal O_X)
\] 
is an isomorphism.
\end{definition}

If $\cal F$ is any coherent sheaf then $\cal F^\vee$ is reflexive, so in particular for any morphism $X\to Y$, the sheaf $\cal T_{X/Y}:=\Omega_{X/Y}^\vee$ is reflexive.  Note that we generally denote $\cal T_{X/k}$ by just $\cal T_X$.  

If $X$ is normal, there is a bijective correspondence between linear equivalence classes of Weil divisors $D$  and rank $1$ reflexive sheaves $\cal O_X(D)$, where addition of Weil divisors corresponds to reflexivised tensor product.  When we need to refer to the class of Weil divisors corresponding to such a sheaf $\cal F$, we will use the notation $[\cal F]$.  Open brackets $|\cal F|$ will refer to the complete linear system of effective divisors in $[\cal F]$.

An open subset $j\colon U\to X$ is called \emph{big} if $\codim_X(X\setminus U)\geq 2$. 

\begin{lemma}\cite[\href{https://stacks.math.columbia.edu/tag/0AY6}{Tag 0AY6}]{stacks}\label{lem:reflexive}
The following conditions are equivalent for a coherent sheaf $\cal F$ on a normal variety $X$:
\begin{enumerate}
\item
$\cal F$ is reflexive,
\item
$\cal F$ is torsion free and has property $S_2$, and
\item
there is a big open subset $j\colon U\to X$ such that $j^*\cal F$ is locally free
(of finite rank)
and $\cal F\cong j_*j^*\cal F$.
\end{enumerate}
\end{lemma}

\subsubsection{Determinants}

The following result will allow us to define the determinant of a torsion-free sheaf that is not necessarily reflexive.

\begin{lemma}[\cite{ishii} Proposition 5.1.7]\label{lem:big_locally_free}
Let $X$ be a normal variety over $k$ and $\cal F$ a torsion-free coherent sheaf on $X$. Then there is a big open subset $j\colon U\hookrightarrow X$ such that $\cal F|_U$ is locally free.
\end{lemma}

\begin{definition}\label{def:det}
	Given a coherent sheaf $\cal F$ on a normal variety $X$, we define a reflexive sheaf $\det\cal F$ called the  \emph{determinant} in increasing generality:
	\begin{enumerate}
		\item\label{itm:det_refl}

If $\cal F$ is reflexive then
	$\det\cal F=(\bigwedge^{\rank\cal F}\cal F)^{\vee\vee}$.
	
	\item\label{itm:det_torsion_free}
If $\cal F$ is torsion free 
then $\det\cal F=i_*\det(\cal F|_U)$,
where $U$ is a big open subset of $X$ on which $\cal F$ is locally free by \autoref{lem:big_locally_free}
\item \label{itm:det_torsion}

If $0\to\cal G\to\cal F'\to \cal F\to 0$ is an exact sequence of sheaves with $\cal G$ and $\cal F'$ torsion free, we define \[\det \cal F=(\det \cal F'\otimes\det\cal G^{-1})^{\vee\vee}.
\]

\end{enumerate}
\end{definition}

\begin{remark}\ 
	
	\begin{itemise}
		\item
	It follows from \autoref{lem:reflexive} that \autoref{itm:det_torsion_free} makes sense and agrees with \autoref{itm:det_refl} when $\cal F$ is reflexive.  \cite[\href{https://stacks.math.columbia.edu/tag/0B38}{Tag 0B38}]{stacks} together with \autoref{lem:reflexive} and \autoref{lem:big_locally_free} 
		implies that \autoref{itm:det_torsion} agrees with \autoref{itm:det_torsion_free} when $\cal F$ is torsion free. 
		\item
		These definitions agree with those which would be obtained by restricting to the big open subset on which all the torsion free sheaves which appear are locally free, then reflexivising.
\end{itemise}
	\end{remark}

\subsubsection{Purely inseparable Galois theory}\label{section:pi_galois}

In this subsection we describe a Galois theory for certain purely inseparable morphisms.  First we clear up some notation surrounding Frobenius.

\begin{definition}
	
	Let $X$ be a scheme over a field of characteristic $p>0$. 
	The \emph{absolute Frobenius} morphism on $X$ is the endomorphism $F_X\colon X\to X$
	induced by the $p\textsuperscript{th}$ power map on rings.  
	
	Now suppose $X$ comes with a morphism $f\colon X\to S$.
	Let $X'=X\times_S S$ the base change of $X$ by $F_S\colon S\to S$.
	The \emph{relative Frobenius} morphism of $X$ over $S$ is the induced $S$-morphism $F_{X/S}\colon X\to X'$.
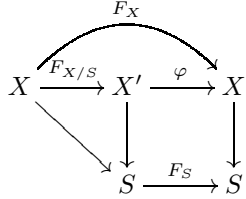
	\[
	\xymatrix{
X\ar^{F_{X/S}}[r]\ar[dr]\ar@/^2.0pc/[rr]^{F_X} & X'\ar[r]^{\phi}\ar[d] & X\ar[d]\\
&S\ar^{F_S}[r] & S	
}\]

	If $k$ is a perfect field and $S=\Spec k$ then $F_k\colon \Spec k\to\Spec k$ is an isomorphism, so $\phi\colon X'\to X$ is also an isomorphism of schemes (but not of $k$-schemes).   Throughout this article, there will usually be a fixed perfect ground field $k$, where $K$ is the function field of a variety over $k$.  In this situation, we denote the relative Frobenius and its target by  $F\colon X\to X^p$.  This is a $k$-morphism and agrees with the global Frobenius $F_X$ up to the isomorphism described above. 
\end{definition}

\begin{definition}\label{defn:height_one_morphism}
	A finite purely inseparable morphism $f\colon X\to Y$
	is said to be of \emph{height one} if there is a finite purely inseparable morphism $g:Y\to X$ such that 
	$g\circ f=F_X$.
\end{definition}

For field extensions this agrees with height as defined in \autoref{not:fields}.
In the case where $k$ is perfect and $f\colon X\to Y$ is a $k$-morphism, there is a $k$-morphism $Y\to X^p$ such that the composition is relative Frobenius of $X$ over $k$.

Let $Y$ be a variety over a perfect field $k$
of characteristic $p>0$.  
A ($k$-)\emph{derivation} is (locally) a $k$-linear map $\Delta\colon \cal O_Y\to\cal O_Y$ satisfying the Leibniz rule.  It can be composed with itself $p$ times, giving a function $\Delta^p$ which we call the $p$\textsuperscript{th} power of $\Delta$.  In characteristic $p>0$, $\Delta^p$ also satisfies the Leibniz rule and so is also a derivation.  
The tangent sheaf $\cal T_{Y/k}$ identifies with the sheaf of $k$-derivations on $Y$.
\begin{definition}
A \emph{foliation} on a variety $Y$ over a perfect field of characteristic $p$ is a coherent saturated subsheaf $\cal F\subset\cal T_Y$ that is closed under Lie brackets and $p$\textsuperscript{th} powers.
\end{definition}
\begin{remark}
The condition of closedness under Lie brackets is in fact redundant by  \cite{gerstenhaber_galois_1964}.  Therefore it is enough to check saturation and $p$-closedness.
\end{remark}

A foliation $\cal F$ of rank $r$ on a variety $Y$ defines a height one purely inseparable $k$-morphism 
$\phi\colon Y\to Z$ of degree $p^{r}$ by taking $Z=\underline{\Spec}_{Y^p}\cal A$,
where $\cal A$ is the subsheaf of $\cal O_Y$ killed by all sections of $\cal F$, which satisfies $\cal O_Y^p\subset\cal A\subset\cal O_Y$.
$Z$ is called the quotient of $Y$ by $\cal F$.

On the other hand, given a height one purely inseparable morphism $Y\to Z$ of $k$-varieties, the sheaf $\cal T_{Y/Z}$ of derivations on $Y$ that vanish on $\cal O_Z\subset\cal O_Y$ is a subsheaf of $\cal T_{Y}$ that is closed under $p$\textsuperscript{th} powers. If $Y$ is normal, then $\cal T_{Y/Z}\subset\cal T_{Y}$ is saturated and hence a foliation on $Y$. Moreover, in the normal case, these procedures give the following correspondence:

\begin{lemma}[{\cite[Proposition 2.4]{ekedahl}}]\label{lem:foliation_height_one}
	Let $Y$ be a normal variety over a perfect field $k$ of characteristic $p>0$.
	There is a one-to-one correspondence between the following three sets:
	\begin{enumerate}
		\item Finite purely inseparable $k$-morphisms $\phi\colon Y\to Z$ of height one to a normal variety,
		\item
		intermediate fields $K(Y)\supset L\supset K(Y)^p$, and
		\item
		foliations $\cal F\subset\cal T_{Y}$.
	\end{enumerate}
\end{lemma}

There is  a canonical bundle formula for quotients by foliations:
\begin{proposition}[{\cite[Corollary 3.4]{ekedahl}, \cite[Proposition 2.10]{pw}}]\label{canonicalbundleformula}
Let $Y\to X$ be a finite purely inseparable morphism of height one of normal varieties over a perfect field $k$ of characteristic $p>0$, and let $\cal F$ be the corresponding foliation. Then
\[ \omega_{Y/X} \cong\det\cal F^{\otimes[p-1]}.
\]
\end{proposition}

\subsection{Grassmannians and linear systems}\label{sec:grassmannians}

Let $S$ be a variety and $\cal E$ a vector bundle on $S$.  We define the Grassmannian $\mathrm{Gr}_S(r,\cal E)$ of $r$-planes on $\cal E$ to be the variety over $S$ which parametrises $r$-dimensional subbundles of $\cal E$.  That is, for a morphism $f\colon X\to S$, giving a diagram
\[
\xymatrix{
X\ar[r]\ar[dr] & \mathrm{Gr}_S(r,\cal E)\ar[d]\\
& S
}
\]
is the same as giving an exact sequence of locally free sheaves
\[\xymatrix{
0\ar[r] & \cal F\ar[r] & f^*\cal E\ar[r] & \cal Q\ar[r] & 0}
\]
with $\rank \cal F=r$.  Furthermore, to give this data, it is sufficient to give the surjection
\[\xymatrix{
f^*\cal E\ar@{->>}[r] & \cal Q}
\]
with $\rank\cal Q=\rank\cal E-r$, for then the kernel is automatically a subbundle.

Furthermore, if we are given a torsion free sheaf $\cal Q$ and a surjection $f^*\cal E\to \cal Q$ then we obtain a rational map $X\dashrightarrow \mathrm{Gr}_S(r,\cal E)$ which is a morphism away from the points where $\cal Q$ is not locally free.  This has codimension at most $2$ by \autoref{lem:big_locally_free}.

Let $X$ be a normal variety with morphism $f\colon X\to S=\Spec(A)$.  Suppose $H^0(X,\cal O_X)=A$, and let $\cal L$ be a  torsion free sheaf of rank one, together with a surjection
	\[\xymatrix{
		\cal O_X^{\oplus d}\ar[r]^-\alpha & \cal L\ar[r] & 0}.
	\]
Then we can define a linear system of Weil divisors associated to the map $\alpha$ by restricting to the complement of the codimension $2$ locus on $X$ where $\cal L$ fails to be locally free, and then taking the closure of the resulting Weil divisors.  This is equivalent to taking the double dual of the above surjection.  Doing this, we obtain a linear system $\frak{M}\subset |\cal O_X(\cal L^{\vee\vee})|$.  We will denote the corresponding subspace of $H^0(X,\cal O_X(\cal L^{\vee\vee}))$ by $V(\frak M)$.  We denote the dimension of $V(\frak M)$ by $\dim_A V(\frak M)$ and the dimension of $\frak M$ by $\dim \frak M$.  So by definition $\dim \frak M=\dim_A V(\frak M)-1$.

\subsection{Flatness and Stein factorisation}

This subsection includes some lemmas which we need later, which are presumably well known but for which we could find no satisfactory reference.

\begin{lemma}\label{lem:flatification}
	Let $g\colon Y\to T$ be a projective dominant morphism of normal varieties.  Then there is an open subset $W\subseteq T$ with $\codim(T\setminus W)\geq 2$ such that $Y_W$ is flat over $W$.
\end{lemma}
\begin{proof}
	By the Raynaud--Gruson flatification (\cite{raynaudgruson} Theorem 5.2.2) there exists a modification $\mu\colon  T'\to T$ such that the proper transform of $g$ under $\mu$ is a flat morphism.  This is defined by the following diagram:
	
	\[\xymatrixcolsep{3pc}\xymatrix{
		Y\times_T T' \\
		\overline{Y\times_T U} \ar@{^{(}->}[u] \ar[r] \ar[d]^-{\text{flat}} & Y \ar[d]^-g \\
		T' \ar[r]^-\mu_-{\text{birational}} & T \\
		\mu^{-1}(U) \ar[r]^-{\mu|_{\mu^{-1}(U)}} \ar@{^{(}->}[u] & U \ar@{^{(}->}[u]}
	\]
where $U$ is an open dense subscheme of $T$ over which $\mu$ is an isomorphism.
	
	Let $W$ be the largest open subset of $T$ over which $\mu$ has finite fibres.  Then $\mu$ is finite over $W$ and moreover by normality of $W$, $\mu$ is an isomorphism over $W$.  Furthermore, as $T'$ is irreducible, $T\setminus W$ must have codimension at least $2$.  But as $\mu$ is an isomorphism over $W$, then $Y_W\times_W{T'_W}\to Y_W$ is an isomorphism.  Thus $Y_W\to W$ is flat.
\end{proof}

\begin{lemma}\label{lem:stein_characterisation}
	Let $f\colon X\to S$ be a proper morphism of irreducible  Noetherian schemes with $X$ normal.  The following conditions are equivalent:
	\begin{enumerate}
		\item\label{itm:push_forward} $f_*\cal O_{X}=\cal O_S$,
		\item\label{itm:integrally_closed} $S$ is integrally closed in $X$, and
		\item\label{itm:function_fields} $S$ is normal and $K(S)$ is integrally closed in $K(X)$.
	\end{enumerate}
\end{lemma}
\begin{proof}
	Let $X\to S'$ be the Stein factorisation of $f$.  Assume \autoref{itm:push_forward}, then $S'=S$ and so \autoref{itm:integrally_closed} follows from \cite[\href{https://stacks.math.columbia.edu/tag/03H0}{Tag 03H0}]{stacks}.  Conversely, assuming \autoref{itm:integrally_closed}  then we also have $S'=S$ by \cite[\href{https://stacks.math.columbia.edu/tag/03H0}{Tag 03H0}]{stacks} and so \autoref{itm:push_forward} follows.  \autoref{itm:function_fields} implies \autoref{itm:integrally_closed} is clear, so we must prove the converse. 
	
	Suppose $S$ is integrally closed in $X$, but $K(S)$ is not integrally closed in $K(X)$.  Then there is some finite morphism $T\to S$ such that $K(S)\to K(T)\to K(X)$.  We then get a rational map $g\colon X\dashrightarrow T\to S$ factoring $X\to S$.  
	
	Let $Y\subset X\times T$ be the closure of the graph of $g$.  It is sufficient to show that $\phi\colon Y\to X$ is an isomorphism.  Suppose that $P$ is a fundamental point of $\phi^{-1}$, so that $\phi^{-1}$ is not defined at $P$.  By Zariski's main theorem, the total transform has dimension at least $1$, and contains a curve $C$.  But then consider the composition $Y\to Y_S\subset X\times S\to X$ will also contract the image of $C$, and so $X\dashrightarrow S$ is not a morphism.  We have a contradiction.
	
\end{proof}

\section{Canonical bundle formula}\label{sec:reinterpret}

The following  notational conventions will be used throughout the paper.

\subsection{Notational conventions}\label{sub:set-up}
	
	\begin{itemise}
		\item
	$f\colon \EuScript{X}\to S$ will be a projective morphism of normal varieties over a perfect field $k$ of characteristic $p>0$, satisfying $f_*\cal O_{\EuScript{X}}=\cal O_{S}$.  
	\item 
	The function field of $S$ will be $K$, and the generic fibre of $f$ will be $X$.  These are the varieties which appear in \autoref{main_theorem}.
	\item
	$T$ will be a normal variety with a purely inseparable height one $k$-morphism $\tau\colon T\to S$ (\autoref{defn:height_one_morphism}).  
	\item 
	We will denote $(\EuScript{X}\times_ST)_{\red}$ by $\cal Z$, and $\cal Y$ will be the normalisation of $\cal Z$.   Often we will refer to $\cal Y$ as the normalisation of $\EuScript{X}\times_{S} T$, with the understanding that this does not imply that $\EuScript{X}\times_S T$ is reduced.
	\item
	Note that the scheme $\EuScript{X}\times_ST$ is reduced if and only if it satisfies the conditions $S_1$ and $R_0$.  However, the  $S_1$ condition is preserved under flat base change by \cite[Theorems 15.1 and 23.3]{matsumura_commutative_1989}, so if $T\to S$ is flat, reducedness of $\EuScript{X}\times_S T$ is equivalent to regularity at its generic point.
	
	\item 
	We will always remove a subset of $S$ of codimension at most $2$ to ensure that $\EuScript{X}\to S$ and $\cal Y\to T$ are flat, using \autoref{lem:flatification}.  Furthermore, we can similarly remove a codimension $2$ set to ensure that $\cal T_{T/S}$ is locally free. The codimension condition ensures this makes no difference in determining Weil divisors.
	\item
	Denote the ideal sheaf of $\cal Z$ in $\EuScript{X}\times_S T$ by  $I_\cal Z$, or just $I$ if no confusion can arise.
	\item
	Denote the morphisms between these schemes as follows:

\[\xymatrix{
\cal Y \ar[r]_-\nu \ar@/^1pc/[rrr]^-\phi \ar[rrd]_-g & \cal Z\ar[r]^-i\ar[rd]^-h& \EuScript{X}\times_S T \ar[r] \ar[d]^-{f_T} & \EuScript{X} \ar[d]^-f \\
& & T \ar[r]^-\tau & S}
\]
	\item
	Let $\cal F_{\cal Y/\EuScript{X}}\subset\cal T_{\cal Y}$ be the foliation corresponding to $\phi$, and let $\cal T_{T/S}\subset\cal T_{T/k}=\cal T_{T/T^p}$ be that corresponding to $\tau$ (see \autoref{lem:foliation_height_one}).  We drop the subscripts whenever no confusion can arise.   
\item The dimension of $\cal X$ will be denoted by $n$, and the rank of $\cal F$ will be denoted by $r$.
	
\item
In addition to denoting the function field of $S$ by $K$, denote that of $T$ by $L$, and drop the calligraphic fonts to denote generic fibres, giving a diagram of schemes over function fields as follows:

\[\xymatrix{
	Y \ar[r]_-\nu \ar@/^1pc/[rrr]^-\phi \ar[rrd]_-g &  Z\ar[r]^-i\ar[rd]^-h& X\otimes_K L \ar[r] \ar[d]^-{f_L} &  X \ar[d]^-f \\
	& & L\ar[r]^-\tau & K}
\]

In particular, we will abuse notation by dropping $\Spec$ in e.g. $\Spec K$, and use $X\otimes_K L$ in place of $X\times_{\Spec K}\Spec L$.

\item  Recall that the Frobenius morphisms considered will all be the relative Frobenius over $\Spec k$, which matches absolute Frobenius up to a canonical isomorphism of the target.

\item The objects defined in the remainder of this section will also be used without further comment throughout the article.
\end{itemise}

\subsection{Construction}\label{sub:pw}

Our construction begins where the proof of \cite[Theorem 3.1]{pw} ends, so we first summarise the constructions made in that proof.  We start with the cotangent sequence
\[\xymatrix{
	g^*\Omega_{T/k} \ar[r] & \Omega_{\cal Y/k} \ar[r] & \ar[r] \Omega_{\cal Y/T} \ar[r] & 0}.
\]

Recall that we assume throughout that we have already replaced  $T$ and $S$ by big open subsets so that $g$ is flat, so we have that $(g^*\Omega_{T/k})^\vee\cong g^*\cal T_{T/k}$.  Thus dualising the left-hand map gives a map $\cal T_{\cal Y/k}\to g^*\cal T_{T/k}$.  It is shown in the proof of point (b) of \cite[Theorem 3.1]{pw} that, if $\cal F:=\cal F_{\cal Y/\EuScript{X}}$ is the foliation corresponding to $\cal Y\to\EuScript{X}$, the composition
\[\cal F\hookrightarrow \cal T_{\cal Y/k}\to g^*\cal T_{T/k}
\]
is injective.
Furthermore, the proof of point (c) of \cite[Theorem 3.1]{pw} also shows that the image of $\cal F$ lands inside $g^*\cal T_{T/S}$, which after shrinking $S$ and $T$ to ensure flatness, is a subsheaf of $g^*\cal T_{T/k}$.
We have already shrunk the base to a big open subset in \autoref{sub:set-up} to ensure that  $\cal T_{T/S}$  is locally free.  
This gives the key injection
\[\cal F\hookrightarrow g^*\cal T_{T/S}.
\]

\begin{definition}
Define $\cal F'_{T/S}$ to be the saturation of $\cal F$ inside $g^*\cal T_{T/S}$, where again we drop the subscript if no confusion can arise.  
\end{definition}

We have an inclusion $\cal F\hookrightarrow \cal F'$ between sheaves of equal rank.  This gives a global element of $\Hom_\cal Y(\det\cal F,\det\cal F')$, which after restricting to a suitable big open set of $\cal Y$ and then reflexivising, gives a canonical global section in $|\det \cal F'-\det\cal F|=|\det(\cal F'/\cal F)|$.  
\begin{definition}
	Define the fixed part of the linear system to be this canonical section, denoted
	\[\frak F_{T/S}\in |\det(\cal F'/\cal F)|.
	\]
	\end{definition}

\begin{remark}
If $\EuScript{X}\times_S T$ is reduced, then $\frak F_{T/S}$ agrees with the canonical choice of $C$ given in (c) of \cite[Theorem 3.1]{pw}, and in that case $\cal F'=g^*\cal T_{T/S}$.
\end{remark}

\begin{definition}
Now define a torsion free sheaf $\cal Q_{T/S}$ by
\[ \xymatrix{
	0 \ar[r] & \cal F'_{T/S} \ar[r] & g^*\cal T_{T/S} \ar^\alpha[r] & \cal Q_{T/S} \ar[r] & 0}.
\]
\end{definition}

\begin{remark}
Let $U$ be an open affine subset of $T$ on which $\cal T_{T/S}$ is free.  Then after restricting $\cal Y$ to $g^{-1}(U)$, $\det \cal Q_{T/S}=(\bigwedge^{n-r}\cal Q_{T/S})^{\vee\vee}$ is globally generated by $\wedge^r\alpha$ away from some codimension $2$ points by \autoref{lem:reflexive}.   
\end{remark}

Finally, let $h\colon \cal Y\dashrightarrow\mathrm{Gr}(r,\cal T_{T/S})$
be the  
rational map induced by $g^*\cal T_{T/S}\to \cal Q_{T/S}\to 0$ (see \autoref{sec:grassmannians}), and let $\cal W_{\Im}$  be the closure of the image.  As $\mathcal{Y}\to\EuScript{X}$ is of height one (\autoref{defn:height_one_morphism}), there is a rational map 
$\EuScript{X}\to\cal Y^p\dashrightarrow \cal W_{\Im}^{p}$. Let $\cal V$ be the normalisation of $\cal W_{\Im}^{p}$ inside $K(\EuScript{X})$.  By definition of normalisation $K(\cal V)$ is integrally closed in $K(\EuScript{X})$, so it must contain $K(S)$ (for it certainly contains $K(S)^{p}$).  Therefore we get a factorisation $\EuScript{X}\dashrightarrow \cal V\to S$.  

\begin{definition}\label{defn:w_and_v}
Here we define the intermediate varieties in the main result \autoref{main_theorem}.  The variety $\cal V$ was defined in the previous paragraph.  Let $\cal W$ be the normalisation of the maximal reduced subscheme of $\cal V\times_S T$.  By universal properties there is a map $\cal Y\dashrightarrow \cal W$.
\end{definition}

\begin{example}
	In the above setting, $\cal W$ and $\cal W_{\Im}$ may be different; see Example~\ref{exmp:pm_fermat_example} with $m>1$.
\end{example}

The following canonical bundle formula is an extension of \cite[Theorem 3.1]{pw}, which we think should be of independent interest.  The first two terms in the relative canonical divisor correspond to the decomposition of the linear system $\mathfrak{C}$ described in \autoref{main_theorem}.  Note that if we allow shrinking $S$, $\det \cal Q_{T/S}$ becomes effective and so we recover the formula of \cite[Theorem 3.1]{pw}.  

	\begin{theorem}\label{thm:C}
In the situation of \autoref{sub:set-up}, recalling that this involved shrinking $S$ and $T$ to big open subsets, we have a linear equivalence of Weil divisors: 
	\[K_{\cal Y}+(p-1)(\frak F_{T/S}+\det Q_{T/S}-g^*\det\cal T_{T/S})\sim\phi^*K_{\EuScript{X}}.
	\]
	
\end{theorem}

\begin{proof}

Firstly note that it makes sense to discuss the pullback of Weil divisors by $\phi$ because it is a finite morphism.  To be precise, we define such a pullback by restricting to the smooth loci of base and target and then taking the closure in $\cal Y$ of the resulting Cartier divisor pullback.  Note that we are also able to remove arbitrary closed subvarieties of codimension at most $2$ from $\EuScript{X}$ prior to calculating this pullback.  Finally, by equidimensionality, the preimage of a  big open subset of $S$ in $\EuScript{X}$ is also a big open subset.  Therefore we are able to remove such a subset without affecting the computation of Weil divisor pullbacks.
Recall that the assumptions of \autoref{sub:set-up} restricted to big open subsets of $S$ and $T$ to assume that $g$ is flat and $\cal T_{T/S}$ is locally free.

To obtain the canonical bundle formula, we use the exact sequences:
\[ \xymatrix{
0 \ar[r] & \cal F \ar[r] & \cal F' \ar[r] & \cal F'/\cal F \ar[r] & 0},
\]
\[ \xymatrix{
0 \ar[r] & \cal F' \ar[r] & g^*\cal T_{T/S} \ar[r] & \cal Q_{T/S} \ar[r] & 0}.
\]

So we have 
\[[-\det\cal F]\sim[\det(\cal F'/\cal F)]-[\det\cal F']\sim [\det(\cal F'/\cal F)]+[\det\cal Q_{T/S}]-[\det g^*\cal T_{T/S}]
\]
which together with \autoref{canonicalbundleformula} gives the required canonical bundle formula.

\end{proof}

\begin{rem}
Note that if the original $f$ has fibres of constant dimension, then the same formula holds over all of $T$, as Weil divisors extend uniquely from the big open subsets of \autoref{sub:set-up}.
\end{rem}

\section{The fixed part of the linear system}\label{sec:universal_property}

In this section we prove \autoref{main_theorem} \autoref{F_normality} and \autoref{F_support}.   First we give a reinterpretation of $\cal F$ and $\cal F'$.

\begin{proposition}\label{prop:compare_F_Omega}
	In the situation of \autoref{sec:reinterpret}, we have
	\[ \cal F=\sheafhom_{\cal Y}(\Omega_{\cal Y/\EuScript{X}},\cal O_\cal Y)\]
	and on the locus where $g\colon \cal Y\to T$ is flat
	\[ \cal F'=\sheafhom_{\cal Y}(\nu^*\Omega_{\cal Z/\EuScript{X}},\cal O_\cal Y).
	\]
\end{proposition}

\begin{proof}
	
	The first part of the proposition holds because 
	\[\cal F=\sheafder_{\cal O_\EuScript{X}}(\cal O_\cal Y,\cal O_\cal Y)=\Omega_{\cal Y/\EuScript{X}}^\vee.\]
	  It remains to prove the statement about $\cal F'$.

Consider the conormal sequence of
 $\cal Z\to\EuScript{X}\times_ST\to\EuScript{X}$:
\[
\xymatrix{
I/I^2 \ar[r] & i^*\Omega_{\EuScript{X}\times_S T/\EuScript{X}}\cong h^*\Omega_{T/S} \ar[r] & \Omega_{\cal Z/\EuScript{X}} \ar[r] & 0}.
\]

Pulling back the above sequence to $\cal Y$,
\begin{equation}\label{eqn:pullback}
\xymatrix{
\nu^*I/I^2 \ar[r] & g^*\Omega_{T/S} \ar[r] & \nu^*\Omega_{\cal Z/\EuScript{X}} \ar[r] & 0}.
\end{equation}

Finally we have the cotangent sequence of $\cal Y\to\cal Z\to\EuScript{X}$:

\begin{equation}\label{eqn:cotangentYZX}
\xymatrix{
\nu^*\Omega_{\cal Z/\EuScript{X}} \ar[r] & \Omega_{\cal Y/\EuScript{X}} \ar[r] & \Omega_{\cal Y/\cal Z} \ar[r] & 0}.
\end{equation}

We break the proof into several claims which together imply the result:
\begin{enumerate}[label=(\roman*)]
	\item
	$\cal F\hookrightarrow(\nu^*\Omega_{\cal Z/\EuScript{X}})^\vee$:
	
	Taking the dual of
	\autoref{eqn:cotangentYZX}
	gives
	\[\xymatrix{
0 \ar[r] & \Omega_{\cal Y/\cal Z}^\vee \ar[r] & \cal F \ar[r] & (\nu^*\Omega_{\cal Z/\EuScript{X}})^\vee }.
	\]
	But $\Omega_{\cal Y/\cal Z}$ is a torsion $\cal O_\cal Y$-module, as it is supported where $\nu$ is not an isomorphism by \autoref{eqn:cotangentYZX}. So $(\Omega_{\cal Y/\cal Z})^\vee=0$, giving the required injection.

\item $(\nu^*\Omega_{\cal Z/\EuScript{X}})^\vee/\cal F$ is torsion:

This is because
$\cal F=(\Omega_{\cal Y/\EuScript{X}})^\vee \to (\nu^*\Omega_{\cal Z/\EuScript{X}})^\vee$
is an isomorphism where $\nu$ is an isomorphism.

	\item
	$(\nu^*\Omega_{\cal Z/\EuScript{X}})^\vee\hookrightarrow g^*\cal T_{T/S}$:
	
	Since $g$ is flat, we have $g^*\cal T_{T/S}=(g^*\Omega_{T/S})^\vee$.
	To show the required injectivity, it is enough to show surjectivity of
	$g^*\Omega_{T/S}\to\nu^*\Omega_{\cal Z/\EuScript{X}}$, which we see from \autoref{eqn:pullback}.
	\item

$(\nu^*\Omega_{\cal Z/\EuScript{X}})^\vee$ is saturated in $g^*\cal T_{T/S}$:
	
	Define $\cal K_{\cal Y}$ to be the kernel in the following exact sequence:
	\[\xymatrix{
	0\ar[r] & \cal K_{\cal Y} \ar[r] & g^*\Omega_{T/S}\ar[r] &\nu^*\Omega_{\cal Z/X}\ar[r] & 0}
	\]
This gives
	\[\xymatrix{
	0 \ar[r] & (\nu^*\Omega_{\cal Z/\EuScript{X}})^\vee\ar[r] & g^*\cal T_{T/S} \ar[r] & \cal K_{\cal Y}^\vee.}
	\]
	The cokernel of
	$(\nu^*\Omega_{\cal Z/\EuScript{X}})^\vee\hookrightarrow g^*\cal T_{T/S}$ is a submodule of the torsion-free $\cal O_\cal Y$-module
	$\cal K_{\cal Y}^\vee$. So it is also torsion-free.

\end{enumerate}

\end{proof}

Next we use this to give an alternate description of $\frak F$.

\begin{proposition}\label{prop:F_support}

For an integral Weil divisor $D$, denote the generic point of $D$ by $\xi_D$.  Then we have:

\[ \frak F=\sum_{D} \length(({\Omega_{\cal Y/\cal Z}})_{\xi_D}) \ D
\]
 
\end{proposition}

\begin{proof}

Recall that the divisor $\frak F$ is defined as the element of $\theta\in\Hom(\det\cal F,\det\cal F')$ which arises from the inclusion $\cal F\hookrightarrow \cal F'$.  We can state this as \[\frak F=\sum_D \mathrm{length}((\det\cal F'/\theta(\det\cal F))_{\xi_d})\ D,\]
where we are using that $D$ is a codimension one point, to ensure that both $\det F'$ and $\det F$ are Cartier divisors there.

On the other hand, in \autoref{prop:compare_F_Omega} we saw that the inclusion $\cal F\hookrightarrow\cal F'$ is dual to the map $\psi$ in the exact sequence
\[\xymatrix{
	\nu^*\Omega_{\cal Z/\EuScript{X}}\ar[r]^-\psi & \Omega_{\cal Y/\EuScript{X}}\ar[r] & \Omega_{\cal Y/\cal Z}\ar[r] & 0}.
\]

On the locus where $\cal Y$ and $\EuScript{X}$ are both regular, the sheaf
$\Omega_{\cal Y/\EuScript{X}}$ is locally free (Theorem of \cite{KimuraNiitsuma} and 38.A of \cite{matsumura}) and hence torsion free.   So the result follows by applying  \autoref{lem:support_dual} to the map $\psi$.

\end{proof}

\begin{lemma}\label{lem:support_dual}
Let $\theta\colon M_1\to M_2$ be a morphism of finitely-generated modules over a Noetherian normal integral domain $A$ such that
	\[\xymatrix{
	(\theta\otimes 1)\colon M_1\otimes_A K(A)\ar[r] & M_2\otimes_A K(A)}
	\]
is an isomorphism, and assume that $M_2$ is torsion free in codimension one.
Then 
\[ \length (M_2/\theta(M_1))_\frak p=
\length\left(\bigwedge^r M_1^\vee / (\bigwedge^r \theta^\vee)(\bigwedge^r M_2^\vee)\right)_\frak p
\]
for every height one prime ideal $\frak p$ of $A$, where $r$ is the rank of the modules $M_i$.
\end{lemma}

\begin{proof}
Let $\frak p$ be a height 1 prime ideal of $A$. Then $A_\frak p$ is a DVR with valuation $v_{\frak p}$, and by the structure theorem for finitely-generated modules over a PID we may write
${M_1}_\frak p\cong \widetilde{M_1}\oplus T_1$ and ${M_2}_\frak p\cong A_\frak p^{\oplus f_2}$, where $\widetilde{M_1}\cong A_\frak p^{\oplus f_1}$ and $T_1\subset{M_1}_\frak p$ is the torsion submodule. Since by hypothesis $M_{1}\otimes_{A}K(A)\cong M_{2}\otimes_{A}K(A)$, we have $f_1=f_2$, so let $r=f_i$.

We may choose bases of $\widetilde{M_1}$ and ${M_2}_\frak p$ so that
$\theta_\frak p|_{\widetilde{M_1}}$ is represented by an $r\times r$ matrix in Smith normal form, which is  diagonal with non-zero entries $a_i$.  The quotient $M_{2\frak p}/\theta(M_{1\frak p})$ has length $\Sigma_{i=1}^r v_{\frak p}(a_i)$.

On the other hand, the dual map $\theta^\vee\colon({M_2}_\frak p)^\vee \to (\widetilde{M_1})^\vee=({M_1}_\frak p)^\vee$ is represented by the same matrix with respect to the dual bases, and
the map $\bigwedge^r(\theta^\vee)$ is multiplication by the determinant of this matrix, so is multiplication by $\prod_{i=1}^r a_i$.
It follows that the length of $\bigwedge^r M_{1\frak p}^\vee / ((\bigwedge^r \theta^\vee)(\bigwedge^r M_{2\frak p}^\vee))$ is $v_{\frak p}(\Pi_{i=1}^r a_i)=\Sigma_{i=1}^r v_{\frak p}(a_i)$ as required.

\end{proof}

Finally we reach the proof of \autoref{main_theorem} \autoref{F_normality} and \autoref{F_support}.

\begin{proposition}[\autoref{main_theorem} \autoref{F_normality}, \autoref{F_support}]\label{cor:almost_conductor_normalisation}
	The support of $\frak F$ is equal to the codimension one part of the locus on which $\nu$ is not an isomorphism.  
	
	In particular,
	$\frak F=0$ if and only if  $(\EuScript{X}\times_ST)_{\red}$ is $R_1$.
\end{proposition}

\begin{proof}
	We have seen in \autoref{prop:F_support} that the support of $\frak F$ is equal to the codimension one part of the support of $\Omega_{\cal Y/\cal Z}$.   Now the result follow from the more  general \autoref{lem:Omega_of_normalisation}. 
\end{proof}

\begin{lemma}\label{lem:Omega_of_normalisation}
	Let $(A,\frak m_A,\kappa_A)$ be a reduced local ring of characteristic $p>0$. Suppose that the normalisation $f\colon A\to B$  is purely inseparable and finite (the latter holds for instance if $A$ is excellent \cite[\href{https://stacks.math.columbia.edu/tag/03GH}{Tag 03GH}]{stacks}). 
	
	Then $\Omega_{B/A}= 0$ if and only if $A$ is normal.
\end{lemma} 

\begin{proof}
	If $\Omega_{B/A}\neq 0$ then $f$ is not an isomorphism and hence $A$ is not normal.  Conversely, suppose that $f$ is not an isomorphism and we will show that $\Omega_{B/A}\neq 0$.
	By hypothesis $\kappa_B/\kappa_A$ is a finite
 purely inseparable extension, where
	$\kappa_B$ denotes the residue field of the local ring $B$.
	If $[\kappa_B:\kappa_A]>1$ then we are done by the surjection
	\[\xymatrix{
		\Omega_{B/A} \ar[r] & \Omega_{\kappa_B/\kappa_A} \ar[r] & 0}.
	\]
	So assume $\kappa_A\to\kappa_B$ is an isomorphism, and consider $\Omega_{(B/\frak m_A B)/\kappa_A}$. If this module is 0 then $B/\frak m_AB$ is an \'etale $\kappa_A$-algebra and hence is isomorphic to a finite product of finite separable extensions of $\kappa_A$. Since $\kappa_A\cong\kappa_B$ it must be finite product of copies of $\kappa_A$, and since $B/\frak m_AB$ is local we conclude that $B/\frak m_AB\cong\kappa_A$. But this is a contradiction by Nakayama, as we would have $B\cong A$. Therefore $\Omega_{(B/\frak m_A B)/\kappa_A}\neq 0$ and so we get the result from the surjection
	\[\xymatrix{
		\Omega_{B/A} \ar[r] & \Omega_{(B/\frak m_A B)/\kappa_A} \ar[r] & 0}.
	\]
\end{proof}

\section{The movable part of the linear system}\label{sec:section_M}

In this section we define and deal with the movable part $\frak M$.

\subsection{Assumption}\label{sub:M_assumption}
In addition to the assumptions in \autoref{sub:set-up}, in order to define $\frak M$ we must now add:
\begin{itemise}
\item The locally free sheaf $\cal T_{T/S}$ must be free, so we will replace $S$ and $T$ by dense open subsets to assume this.
\end{itemise}

\subsection{Construction}\label{sub:construct_M}

Let $U$ be the big open subset of $\cal Y$ on which $\cal Q$ and $\cal F'$ are locally free.

We have an exact sequence 
\[\xymatrix{
0\ar[r] & \cal F' \ar[r] & g^*\cal T_{T/S} \ar[r] & \cal Q\ar[r] & 0}.
\]

Dualising this gives 
\[\xymatrix{
0 \ar[r] & \cal Q^{\vee} \ar[r] & g^*\Omega_{T/S} \ar[r] & (\nu^*\Omega_{\cal Z/\EuScript{X}})^{\vee\vee}}.
\]

Furthermore, this  is right exact after restriction to  $U$  since $\cal Ext^1(\cal Q^{\vee}|_U,\cal O_{U})=0$.
Recall that $r=\rank\cal F=\rank\nu^*\Omega_{\cal Z/\EuScript{X}}$.

As $\cal F'|_U$ is locally free, $\cal F'^\vee|_U=(\nu^*\Omega_{\cal Z/\EuScript{X}})^{\vee\vee}|_U$ is also locally free, and we have $\det(\nu^*\Omega_{\cal Z/\EuScript{X}})^{\vee\vee}|_U\cong(\bigwedge^r (\nu ^*\Omega_{\cal Z/\EuScript{X}})^{\vee\vee})|_U $.  Note also that to determine a linear system of Weil divisors, it is enough to define it after restriction to the big open subset $U$.

\begin{definition}
The movable linear system $\frak M_{T/S}\subset | \det((\nu^*\Omega_{\cal Z/\EuScript{X}})^{\vee\vee}) |$ of Weil divisors on $\cal Y$ is defined by the map
\[\xymatrix{
\bigwedge^r g^*\Omega_{T/S}|_U \ar[r] & \bigwedge^r (\nu ^*\Omega_{\cal Z/\EuScript{X}})^{\vee\vee}|_U \ar[r] &  0}.
\]
\end{definition}

As $U$ is a big open set, each divisor in the linear equivalence class extends uniquely to a divisor on $\cal Y$, and the linear equivalence class is equal to
\[\det((\nu ^*\Omega_{\cal Z/\EuScript{X}})^{\vee\vee})=-\det(\cal Q^\vee)=\det\cal Q.
\]

We now ensure that this is consistent with our earlier description of the varieties $\cal W$ and $\cal V$ (\autoref{defn:w_and_v}).
 
 \begin{lemma}\label{lem:plucker}
 		Let $U\subseteq\cal Y$ be the big open subset on which $\cal Q$ and $\cal F'$ are locally free.
 Then 
 	the rational maps induced by $g^*\cal T_{T/S}\to \cal Q\to 0$ and $\bigwedge^r g^*\Omega_{T/S}|_U\to (\bigwedge^r \Omega_{\cal Y/\EuScript{X}})^{\vee\vee}|_U\to 0$ have isomorphic images.
 \end{lemma}
 
 \begin{proof}
 We may replace $\cal Y$ with $U$ to assume that $\cal Q$ is locally free.  Then the dual of the sequence 
 \[\xymatrix{
0\ar[r] & \cal F' \ar[r] & g^* \cal T_{T/S} \ar[r] & \cal Q \ar[r] & 0}
 \]
 is exact:
  \[\xymatrix{
0\ar[r] & \cal Q^{\vee} \ar[r] & g^*\Omega_{T/S} \ar[r] & (\nu^*\Omega_{\cal Z/\EuScript{X}})^{\vee\vee} \ar[r] & 0}.
 \]
 	
 	The map $ \bigwedge^{r} \cal F'\to \bigwedge^{r} g^*\cal T_{T/S}$ determines composition of $\cal Y\dashrightarrow \cal W_{\Im}$ with the Pl\"ucker embedding of $\mathrm{Gr}_{T}(r, \cal T_{T/S})$.  If we then compose this with the dual Grassmannian isomorphism, then using the identifications
 	\begin{gather*}\left(\bigwedge^{r}g^*\Omega_{T/S}\to \bigwedge^{r}\nu^*\Omega_{\cal Z/\EuScript{X}}\right)^\vee = \left(\left(\bigwedge^{r}\cal \nu^*\Omega_{\cal Z/\EuScript{X}}\right)^\vee\to \left(\bigwedge^{r}g^*\cal T_{T/S}\right)\right) \\
	=\left(\bigwedge^{r}\cal F'\to \bigwedge^{r}g^*\cal T_{T/S}\right)
	\end{gather*}
	the resulting map $\cal Y\dashrightarrow\mathrm{Gr}_T\left( {d \choose r}-1, \bigwedge^r \Omega_{T/S} \right)$ is exactly the one given by $\frak M_{T/S}$.
 \end{proof}

Note that the map $g^*\Omega_{T/S}\to (\nu^*\Omega_{\cal Z/\EuScript{X}})^{\vee\vee}$ matches the map $h^*\Omega_{T/S}\to \Omega_{\cal Z/\EuScript{X}}$ on the non-empty open subset of $\cal Y$ on which $\cal Y\to\cal Z$ is an isomorphism.  Therefore, as we are interested in a movable linear system, it is enough to study it on this open subset, for there is a unique extension to all of $\cal Y$.  Therefore we now shift attention to $\cal Z$.

\subsection{The essential part of a base change}

In this subsection we prove \autoref{main_theorem} \autoref{M_reducedness} via \autoref{thm:trick}.  
We begin with some preparatory lemmas.

First note that
in one direction of \autoref{main_theorem} \autoref{M_reducedness} it is easy to obtain a stronger statement:

\begin{lemma}\label{lem:Q_zero}
	
	If $\EuScript{X}\times_S T$ is regular in codimension zero, then $\cal Q_{T/S}=0$.

\end{lemma}
\begin{proof}
	The exact sequence
	\[\xymatrix{
		\nu^*I/I^2\ar[r] & g^*\Omega_{T/S}\ar[r] & \nu^*\Omega_{\cal Z/\EuScript{X}}\ar[r] & 0}
	\]
	dualises to
	\[\xymatrix{
		0\ar[r] & \cal F'\ar[r] & g^*\cal T_{T/S}\ar[r] & (\nu^*I/I^2)^\vee}.
	\]
	So recalling that we have assumed $g\colon \cal Y\to T$ to be flat, we get the natural inclusion
	\[\cal Q_{T/S}\subseteq (\nu^*I/I^2)^\vee.
	\]

	The statement follows because if $\EuScript{X}\times_ST$ is $R_0$ then $I/I^2$ is torsion and so the right-hand side is zero.
\end{proof}

Next is a useful criterion for when a base change is non-reduced.

\begin{lemma}\label{lem:normalised_base_change_degree}
	In the situation of \autoref{sub:set-up}, $\EuScript{X}\times_ST$ is non-reduced if and only if ${\deg(\cal Y\to\EuScript{X})}<{\deg(T\to S)}$.
\end{lemma}
\begin{proof}
First recall that reducedness is equivalent to having properties $R_0$ and $S_1$.
Since $T\to S$ is a universal homeomorphism, to show that Serre's condition $S_k$ is preserved we may reduce to the local case. Since $\EuScript{X}\times_ST\to\EuScript{X}$ is flat, then \cite[Theorems 15.1 and 23.3]{matsumura_commutative_1989} imply that $\EuScript{X}\times_ST$ is $S_1$ because $\EuScript{X}$ is.
So it remains to show that $\EuScript{X}\times_ST$ is not $R_0$ if and only if the degree inequality is strict.
Since the extension $K(T)/K(S)$ is purely inseparable, $K(\EuScript{X})\otimes_{K(S)}K(T)$ is an Artinian local algebra, and its dimension as a $K(\EuScript{X})$-module is exactly $p^{\deg(T\to S)}$  because $K(T)$ is a flat $K(S)$-module. Since the residue field of this algebra is $K(\cal Y)$, then it contains a copy of $K(\cal Y)$ by the Cohen structure theorem, and the containment is strict if and only if the base change is non-reduced.
\end{proof}

Finally, we prove that $\cal Z\to T$ is its own Stein factorisation:

\begin{lemma}\label{lem:red_is_stein}
	In the situation of \autoref{sub:set-up},
	$h_*\cal O_\cal Z=\cal O_T$.
\end{lemma}

\begin{proof}
	We may assume that $S$ and $T$ are affine.
	Since $f_*\cal O_{\EuScript{X}}=\cal O_S$ and $\EuScript{X}\to S$ is assumed flat, it follows that ${f_T}_*\cal O_{\EuScript{X}\times_S T}=\cal O_{T}$ by \cite[\href{https://stacks.math.columbia.edu/tag/03GY}{Tag 03GY}]{stacks}.  So if $A$ is a homogeneous coordinate ring of $\EuScript{X}\times_S T$, then in the zeroth graded piece, $A_0=\cal O_T$ is reduced. Now $\cal Z$ is realized as $\Proj(A/\sqrt{0})$, and the quotient map $A\to A/\sqrt{0}$ is a graded homomorphism that induces an isomorphism in grading 0. Therefore $(A/\sqrt{0})_0=\cal O_T$; that is, $H^0(\cal Z,\cal O_\cal Z)=H^0(T,\cal O_T)$.
\end{proof}

We now prove a version of \autoref{thm:trick} for fibrations, from which \autoref{thm:trick} follows immediately by restriction to generic fibres. 
The purpose of \autoref{thm:trick} is to reduce the study of geometrically non-reduced varieties to the case where $H^0(\cal Z,\Omega_{\cal Z/\EuScript{X}}^\vee)= 0$.
We see from the sequence
\[\xymatrix{
	0 \ar[r] & \Omega_{\cal Z/\EuScript{X}}^\vee \ar[r] & h^*\cal T_{T/S} \ar[r] & (I/I^2)^\vee}
\]
that on such a variety we also have
\[H^0(\cal Z, h^*\cal T_{T/S})\hookrightarrow H^0(\cal Z,(I/I^2)^\vee).
\]

\begin{proposition}[\autoref{thm:trick}]\label{trick}
With notation as in \autoref{sub:set-up},
after possibly replacing $S$ and $T$ by dense open subschemes
there exist normal $k$-varieties and $k$-morphisms fitting into the following commutative diagram
\[	\xymatrix{
		\cal Y\ar[r]\ar[d] &\cal Y'\ar[r]\ar[d] & \EuScript{X}\ar[d]\\
		T\ar[r] & T'\ar[r] &S	
	}
\]
such that the following hold:
\begin{enumerate}
	\item
	$\cal Y'$ is the normalisation of $\cal Z':=(\EuScript{X}\times_S T')_{\mathrm{red}}$.
\item\label{itm:same_deg_vars}
$\cal Y\to\cal Y'$ and $T\to T'$ have the same degree.
\item\label{itm:subext_factors_vars}  If $T\to T''\to S$ satisfies property \autoref{itm:same_deg_vars} with respect to $\cal Y''=(\EuScript{X}\times_S T'')_{\mathrm{red}}^{\nu}$, then there is a factorisation
\[T\to T''\to T'\to T.
\]
\item
$\cal Y'\times_{T'}T$ is reduced and $\cal Y$ is its normalisation.
	\item\label{itm:omega_zero_vars}
$h'_*(\Omega_{\cal Z'/\EuScript{X}}^\vee )=0$, 
where $h'\colon\cal Z'\to T'$.

\end{enumerate}

Furthermore, $T'$ is the unique subextension satisfying both \autoref{itm:same_deg_vars} and \autoref{itm:omega_zero_vars}, and the unique subextension satisfying both  \autoref{itm:same_deg_vars} and \autoref{itm:subext_factors_vars}.
\end{proposition}

\begin{proof}

We may replace $S$ and $T$ by open subsets to assume that $T=\Spec A$ is 
affine, $\Omega_{T/S}\cong\cal O_T^{\oplus d}$ is free, and
$g\colon\cal Z\to T$ is flat.
We will show that, after shrinking the base, $h_*(\Omega_{\cal Z/\EuScript{X}}^\vee)$ is a foliation on $T$, and it induces a foliation $\cal H$ on $\cal Y$. We then show that the varieties obtained by taking the quotients of $T$ and $\cal Y$ by these respective foliations satisfy the desired properties.

Pushing forward the inclusion $\Omega_{\cal Z/\EuScript{X}}^\vee\hookrightarrow h^*\cal T_{T/S}$ yields an injection
\[h_*(\Omega_{\cal Z/\EuScript{X}}^\vee) \hookrightarrow h_*h^*\cal T_{T/S}\cong\cal T_{T/S} .
\]
where the isomorphism uses the fact that $\cal T_{T/S}$ is free, and $h_*\cal O_{\cal Z}=\cal \cal O_T$ by \autoref{lem:red_is_stein}.

\textbf{Step 1:}\label{saturationispclosed}
$\cal G:=h_*(\Omega_{\cal Z/\EuScript{X}}^\vee)\subset\cal T_{T/S}$ is a foliation, inducing a morphism $T\to T'$.

From the sequence
\[\xymatrix{
0 \ar[r] & \Omega_{\cal Z/\EuScript{X}}^\vee \ar[r] & h^*\cal T_{T/S} \ar[r] & (I/I^2)^\vee}
\]
we see that $\Omega_{\cal Z/\EuScript{X}}^\vee \subset h^*\cal T_{T/S}$ is a saturated subsheaf, because its cokernel is a subsheaf of a torsion free sheaf. Therefore 
$h_*(\Omega_{\cal Z/\EuScript{X}}^\vee)$ is saturated in $h_*h^*\cal T_{T/S}$.
Since closedness under Lie brackets follows from $p$-closedness
\cite{gerstenhaber_galois_1964},
we only need to show the $p$-closedness of $\cal G$.  Recall that we assume that $T=\Spec A$.

Let $\Delta\in H^0(T,\cal G)\subset H^0(T,h_*h^*\cal T_{T/S})=H^0(\cal Z,h^*\cal T_{T/S})$.
We first unravel what it means to say $\Delta\in H^0(T,\cal G)$.
Let $\{U_i=\Spec B_i\}$ be an affine open cover of $\cal Z$.
Then $\Delta$ corresponds to a compatible collection of
\[\Delta_i\in\Gamma(U_i,\Hom_{B_i}(\Omega_{\cal Z/\EuScript{X}}|_{U_i},B_i))=\Der_{\cal O_\EuScript{X}}(B_i,B_i)
\]
and since they glue, the images of the $\Delta_i$ when applied to a fixed element of $A=H^0(\cal Z,\cal O_{\cal Z})\subset B_i$ also glue to give an element of $H^0(\cal Z,\cal O_\cal Z)=A$.
Now each $\Delta_i^p\in\Der_{\cal O_\EuScript{X}}(B_i,B_i)$, and the collection $\{U_i,\Delta_i^p\in\Gamma(U_i,\Hom_{B_i}(\Omega_{\cal Z/\EuScript{X}}|_{U_i},B_i)\}$ still glues
since the following diagram commutes.
\[\xymatrixcolsep{4pc}\xymatrix{
\cal O_{U_i} \ar[r]^-{\Delta_i} \ar[d] & \cal O_{U_i} \ar[r]^-{\Delta_i} \ar[d] & \cdots \ar[r]^-{\Delta_i} & \cal O_{U_i} \ar[d] \\
\cal O_{U_i\cap U_j} \ar[r]^-{\Delta_i|_{\cal O_{U_i\cap U_j}}} & \cal O_{U_i\cap U_j} \ar[r]^-{\Delta_i|_{\cal O_{U_i\cap U_j}}} & \cdots \ar[r]^-{\Delta_i|_{\cal O_{U_i\cap U_j}}} & \cal O_{U_i\cap U_j}}
\]
So we obtain a section  of $H^0(T,\cal G)$ which agrees with $\Delta^p$ as a function on $\cal O_T$.

	\textbf{Step 2:}	The image $\cal H$ of  $h^*h_*\Omega_{\cal Z/\EuScript{X}}^{\vee}\to \Omega_{\cal Z/\EuScript{X}}^{\vee}$ is $p$-closed.

		Let $U=\Spec(B)$ be an affine open subset of $\cal Z$.  Then $\Omega_{\cal Z/\EuScript{X}}^{\vee}(U)=\Der_{\cal O_{\EuScript{X}}}(B,B)$.  As in the proof of the previous lemma, a section coming from $h^*h_*(\Omega_{\cal Z/X}^{\vee})$ is a compatible system of derivations on charts, which glue to give a global section.  We saw above that this means their $p$-powers also glue to give a global section.  It follows from this that if $\Delta$ is in the image of  $h^*h_*(\Omega_{\cal Z/\EuScript{X}}^{\vee})\to \Omega_{\cal Z/\EuScript{X}}^{\vee}$ then so is $\Delta^p$.

\textbf{Step 3:}  Obtaining a variety $\cal Y'$ fitting into the following diagram:
	\[
	\xymatrix{
		\cal Y\ar[r]\ar[d] &\cal Y'\ar[r]\ar[d] & \EuScript{X}\ar[d]\\
		T\ar[r] & T'\ar[r] &S	.
	}
	\]

Let $U$ be the open subset on which $\cal Y\cong \cal Z$.  Then $\cal H$ induces a morphism $U\to U'\to \EuScript{X}$, from which we can take normalisations of $\EuScript{X}$ in the corresponding function fields, to get a normal variety $\cal Y'$ that factors the morphism $\cal Y\to \cal Y'\to \EuScript{X}$.
The middle vertical arrow exists because $\cal O_{T'}$ is killed by $\cal G$, and so its image in $\cal O_U$ is also killed by the image of $h^*\cal G$  in $\Omega^{\vee}_{\cal Z/\EuScript{X}}(U)$.

\textbf{Step 4:} $\deg(\cal Y\to \cal Y')=\deg(T\to T')$.

Consider the composition
 \[h^*h_*\Omega^{\vee}_{\cal Z/\EuScript{X}}\twoheadrightarrow \cal H\hookrightarrow \Omega_{\cal Z/\EuScript{X}}^\vee\hookrightarrow h^*\cal T_{T/S}.
 \]

As $h^*h_*\Omega_{\cal Z/\EuScript{X}}^\vee\to h^*\cal T_{T/S}$ is injective by flatness, we find that $h^*h_*\Omega_{\cal Z/\EuScript{X}}^\vee\cong \cal H$, and so $\deg(\cal Y\to \cal Y')=\deg(T\to T')=p^{\mathrm{rank}{\cal H}}$.

\textbf{Step 5:} We claim that both squares in the diagram of Step 3 are normalised base changes and $\cal Y'\times_{T'}T$ is reduced.

Let $\tilde{\cal Y'}$ be the normalisation of $\EuScript{X}\times_S T'$.  By universal properties this fits into a similar diagram as above, where both squares are normalised base changes, so it suffices to show that $\cal Y'=\tilde{\cal Y'}$.  By universal properties there is a morphism $\cal Y'\to \tilde{\cal Y'}$, but as $\cal Y$ is the normalised base change of $\tilde{\cal Y'}$ by $T$, the morphism $\cal Y\to \tilde{\cal Y'}$ can have degree at most that of $T\to T'$.  Therefore $\cal Y'\to \tilde{\cal Y'}$ is an isomorphism as both are normal and the morphism between them is forced to have degree at most $1$.  That $\cal Y'\times_{T'}T$ is reduced follows from \autoref{lem:normalised_base_change_degree}.

\textbf{Step 6: } Our varieties satisfy part \autoref{itm:subext_factors_vars}.

Suppose 	\[
\xymatrix{
	\cal Y\ar[r]\ar[d] &\cal Y''\ar[r]\ar[d] & \EuScript{X}\ar[d]\\
	T\ar[r] & T''\ar[r] &S	
}
\]
is a diagram of normalised base changes such that $\deg(\cal Y\to \cal Y'')=\deg(T\to T'')$.  Then let $\cal G''$ be the foliation determining $T\to T''$ and $\cal H''$ be the foliation determining $\cal Y\to \cal Y''$.  We claim that $\cal G''\subseteq \cal G$ and $\cal H''\subseteq \cal H$.

Let 
\[\xymatrix{
	\cal Z\ar[r]\ar[d] &\cal Z''\ar[r]\ar[d] & \EuScript{X}\ar[d]\\
	T\ar[r] & T''\ar[r] &S	
}
\]
be the corresponding diagram of reduced base changes.  Note that it follows from the assumptions on degrees that $\cal Z=\cal Z''\times_{T''} T$, that is to say this base change is already reduced.  This means $\Omega_{\cal Z/\cal Z''}=h^*\Omega_{T/T''}$, which is trivial (after shrinking the base).  The surjection 
$
\Omega_{\cal Z/\EuScript{X}}\twoheadrightarrow \Omega_{\cal Z/\cal Z''}
$
dualises to give an injection
$\Omega_{\cal Z/\cal Z''}^\vee\hookrightarrow  \Omega_{\cal Z/\EuScript{X}}^\vee$.

This means that $h_*(\Omega_{\cal Z/\cal Z''}^\vee)=\cal G''\subset \cal G$, and hence we also have that $\cal H''\subset \cal H$ (after restricting to a suitable open subset and taking saturations).
 
So that there is a factorisation
\[\xymatrix{
	\cal Y\ar[r]\ar[d] &\cal Y''\ar[r]\ar[d] &\cal Y'\ar[r]\ar[d]& \EuScript{X}\ar[d]\\
	T\ar[r] & T''\ar[r] & T'\ar[r]&S	
}
\]
from which we get

\[\xymatrix{
	\cal Z\ar[r]\ar[d] &\cal Z''\ar[r]\ar[d] &\cal Z'\ar[r]\ar[d]& \EuScript{X}\ar[d]\\
	T\ar[r] & T''\ar[r] & T'\ar[r]&S	
}\]
where the left and middle squares are Cartesian.  By a similar argument to above, this means that $H^0(\cal Z'',\Omega_{\cal Z''/\EuScript{X}}^\vee)\neq 0$ so long as $\cal Z''\neq \cal Z'$, as the morphism $\cal Z''\to \cal Z'$ determines such a global section.

\textbf{Step 7:} Our varieties satisfy part \autoref{itm:omega_zero_vars}.
 
Suppose that $h_*(\Omega^\vee_{\cal Z'/\EuScript{X}})\neq 0$.  Then by applying steps 1-5, we obtain a morphism $T'\to T''$ with positive degree such that $\cal Y''\times_{T''} T$ is reduced.  But by step 6 this implies that there is a morphism $T''\to T'$ such that the composition factors into $T\to S$, and so $T''=T'$, giving a contradiction.

\textbf{Step 8:}  The ``furthermore."

It is clear that there can be at most one morphism satisfying both \autoref{itm:same_deg_vars} and \autoref{itm:subext_factors_vars}.  We have shown the existence of this already.   Now suppose we have another morphism $T\to T''$ satisfying \autoref{itm:same_deg_vars} and \autoref{itm:omega_zero_vars}.  Then by \autoref{itm:same_deg_vars} we obtain a factorisation $T\to T''\to T'$.  But then $\deg(\cal Y''\to \cal Y')=\deg(T''\to T')$, and as in the argument of step 6, $h''_*(\Omega_{\cal Z''/\EuScript{X}}^\vee)\neq 0$ unless this degree is $1$, as required.

\end{proof}

\begin{proposition}[\autoref{main_theorem} \autoref{M_reducedness}]\label{prop:M=0_iff_reduced}
In the situation of \autoref{sub:set-up}, after possibly replacing $S$ and $T$ by open subschemes, we have
$\frak M_{T/S}=0$ $\iff$ $\EuScript{X}\times_S T$ is reduced.
\end{proposition}

\begin{proof}
	If $\EuScript{X}\times_S T$ is reduced, then it follows from \autoref{lem:Q_zero} that $\frak M=0$.  So suppose that $\EuScript{X}\times_S T$ is non-reduced.  	Apply \autoref{trick} to obtain $\cal Y'$ and $T'$ such that $\cal Y'\times_{T'}T$ is reduced.
	The linear system $\frak M_{T/S}\subset |\det((\nu^*\Omega_{\cal Z/\cal X})^{\vee\vee})|=|-\det\cal F'_{T/S}|$ being 0 is equivalent to the line bundle $-\det\cal F'_{T/S}|_U$ being trivial, where $U$ is the big open subset from \autoref{sub:construct_M}, and similarly for $T'/S$. Now $-\det\cal F'_{T/S}\geq -\psi^*\det\cal F'_{T'/S}$ by \autoref{lem:compare_Qs}, so $\frak{M}_{T'/S}\neq 0$ implies that $\frak{M}_{T/S}\neq0$. So we may replace $\cal Y$ and $T$ with  $\cal Y'$ and $T'$ to assume that  $H^0(\cal Z,\Omega_{\cal Z/\EuScript{X}}^\vee)=0$.

	Let $\cal R$ be the cokernel of $0\to \Omega_{\cal Z/\EuScript{X}}^\vee\to h^*\cal T_{T/S}$.  The map $h^*\cal T_{T/S}\to \cal R$ matches the map $g^*\cal T_{T/S}\to\cal Q$ on the open subset $U$ on which $\cal Y\cong\cal Z$. Note that the generic rank of $\cal R$ is less than that of $\cal T_{T/S}$.

		We see that  $H^0(\cal Z, h^*\cal T_{T/S})\to H^0(\cal Z, \cal R)$ is injective, and $\cal R$ is torsion free because it is a subsheaf of $(I/I^2)^\vee$.  Therefore for any open subset $U$ of $\cal Z$, the map $h^*\cal T_{T/S}|_U\to R|_U\to 0$ cannot be a quotient by a subsheaf that pulled back from $T$.  Therefore the rational map 
		 $\cal Z\dashrightarrow \cal W_{\Im}'$ induced by $h^*\cal T_{T/S}\to \cal R\to 0$ is non-trivial.  This matches the rational map given by $g^*\cal T_{T/S}\to \cal Q\to 0$ as these maps match on an open subset, which in turn matches the rational map induced by $\frak M$ by \autoref{lem:plucker}.  Therefore $\frak M\neq 0$.	
\end{proof}

\begin{proposition}\label{lem:compare_Qs}
	Let $f\colon \EuScript{X}\to S$ be a morphism of normal varieties with equidimensional fibres, and let $T\to T'\to S$ be a height one composition of morphisms of normal varieties.  Let the following be the corresponding normalised base changes and their associated foliations:
	\[\xymatrixcolsep{5em}\xymatrix{
		\cal Y \ar[d]^-g \ar[r]^{\psi}_{\cal F_{T/T'}} \ar@/^1.5pc/[rr]^{\phi}_{\cal F_{T/S}} & \cal Y' \ar[r]_{\cal F_{T'/S}} \ar[d]^-{g'} & \EuScript{X} \ar[d]^-f \\
		T \ar[r] & T' \ar[r] & S}.
	\]
	Let $\cal F'_{T/T'}$ denote the saturation of $\cal F_{T/T'}$ in $g^*\cal T_{T/T'}$, and similarly for $\cal F'_{T/S}$ and $\cal F'_{T'/S}$. Then
	\[ |\det\cal F'_{T/T'}+\psi^*\det\cal F'_{T'/S}-\det\cal F'_{T/S}|\neq \emptyset.
	\]
\end{proposition}

\begin{proof}
	
	Note that we are free to remove a closed subvariety of codimension at most two from $S$, $T$ and $T'$, and also from $\cal Y,\cal Y'$, and $\EuScript{X}$ (by equidimensionality) to assume that $g$, $g'$, and $\psi$ are flat. The flatness of $\psi$ is attainable because $\cal Y$ becomes regular and $\cal F_{T/T'}$ locally free after removing a closed subset of codimension at least 2, and in this situation $\psi$ is flat by Proposition 2.1.4 of~\cite{maddock}.
	
	We claim there is a diagram of exact sequences as follows:
	
	\[\xymatrix{
		0\ar[r] & \cal F_{T/T'}\ar@{^{(}->}[d]\ar[r]^-{a_1} & \cal F_{T/S} \ar[r]^-{b_1} \ar@{^{(}->}[d] & \psi^*\cal F_{T'/S}\ar@{^{(}->}[d] \ar[r] &0  \\
		0\ar[r] & \cal F'_{T/T'}\ar@{^{(}->}[d]\ar[r]^-{a_2} & \cal F'_{T/S} \ar[r]^-{b_2} \ar@{^{(}->}[d] & \psi^*\cal F'_{T'/S}\ar@{^{(}->}[d]  \\
		0\ar[r] & g^*\cal T_{T/T'}\ar[r] & g^* \cal T_{T/S} \ar[r] & \psi^*g'^*\cal T_{T'/S}  \ar[r] & 0}
	\]

	The proof of Theorem 3.11 in \cite{pw} gives exactness of the top and bottom rows of the diagram, so we only need to show that the middle row fits in.

	 $\psi^*\cal F_{T'/S}'$ is saturated in $\psi^*g'^*\cal T_{T'/S}$ by \autoref{lem:pullback_saturated_saturated}, and therefore $\psi^*\cal F_{T'/S}'$ is the saturation of $\psi^*\cal F_{T'/S}$ in $\psi^*{g'}^*\cal T_{T'/S}$. The left and middle exactness of the second row then follow from the exactness of the first:

	\begin{itemize}
		\item
		(Left exactness)
		If $\delta\in\Ker(\cal F'_{T/T'}\to\cal F'_{T/S})$ then $f\delta\in\Ker(\cal F_{T/T'}\to\cal F_{T/S})$ for some nonzero $f\in\cal O_\cal Y$, and so $\delta=0$ because $\cal F'_{T/S}$ is torsion free.
		\item
		(Middle exactness)
		If $\delta\in\Im(\cal F'_{T/T'}\to\cal F'_{T/S})$ then $f\delta\in\Im(\cal F_{T/T'}\to\cal F_{T/S})$ for some nonzero $f\in\cal O_\cal Y$, and so $0=b_1(f\delta)=f\cdot b_2(\delta)$. 
		
		Conversely, if $\Delta\in\Ker(b_2)$ then $g\Delta\in\Ker(b_1)=\Im(\cal F_{T/T'}\to\cal F_{T/S})$ for some nonzero $g\in\cal O_\cal Y$. Then there is an element $\delta\in\cal F_{T/T'}$ such that $a_1(\delta)=g\Delta$.  But $\delta\in\cal F'_{T/T'}$ and so $a_2(\delta)=g\Delta$.  But $\cal F'_{T/T'}$ is saturated in $g^*\cal T_{T/S}$, and so is saturated in $\cal F'_{T/S}$ too.  That means there must be $\delta'$ such that $a_2(\delta')=\Delta$.
	\end{itemize}
	
	Let $\cal G$ denote the image of $\cal F'_{T/S}$ in $\psi^*\cal F'_{T'/S}$.
	Since $\psi^*\cal F_{T'/S} \hookrightarrow\cal G\hookrightarrow\psi^*\cal F'_{T'/S}$, the ranks of $\cal G$ and $\psi^*\cal F'_{T'/S}$ are the same,
	and so the inclusion $\cal G\hookrightarrow\psi^*\cal F'_{T'/S}$ gives an element of
	$\Hom_\cal Y(\det\cal G,\psi^*\det\cal F'_{T'/S})\cong H^0(\cal Y,\psi^*\det\cal F'_{T'/S}-\det\cal G)$. Therefore
	\[ \det\cal F'_{T/S}=\det\cal F'_{T/T'}+\det\cal G\leq\det\cal F'_{T/T'}+\psi^*\det\cal F'_{T'/S}.
	\]
\end{proof}

\section{The universal property}\label{sec:morphism}

In this section we prove \autoref{fibre_reduced}, \autoref{maximal_reduced}, and \autoref{pullback_from_w} of \autoref{main_theorem}.  First we verify that we can replace $\EuScript{X}$ by a birational model, and in particular pass to a resolution of indeterminacies.

Recall, using \autoref{defn:w_and_v} and \autoref{lem:plucker}, that the varieties $\cal V$ and $\cal W$ are defined as follows. The movable linear system $\frak M_{T/S}$ on $\cal Y$ of \autoref{sec:section_M} defines a rational map $\cal Y\dashrightarrow\cal W_{\Im}$, and $\cal V$ is defined as the Stein factorisation of $\EuScript{X}\dashrightarrow \cal W_{\Im}^{p}$. Then $\cal W$ is defined as the normalisation of the maximal reduced subscheme of $\cal V\times_S T$.

\begin{proposition}\label{prop:assume_morphisms}
	The birational equivalence class of the rational map $\EuScript{X}\dashrightarrow \cal V$ and linear system $\frak M$ is independent of the choice of models $\EuScript{X}/S$ within its birational equivalence class.
\end{proposition}

\begin{proof}
	
	The construction is local on the base, so it is clear that the birational equivalence class of $\cal V$ is unaffected by  birational changes in $S$.  	
	Now given any two birational varieties $\EuScript{X}_1$ and $\EuScript{X}_2$ over $S$, we can let $\widetilde{\EuScript{X}}$ be the normalisation of the graph of $\EuScript{X}_1\dashrightarrow\EuScript{X}_2$, to produce a normal variety with birational morphisms to both.  Therefore it is enough to fix a model $\widetilde{\EuScript{X}}$ of $\EuScript{X}$, which is a resolution of indeterminacies of $\EuScript{X}\dashrightarrow \cal V$, and show that applying the construction of \autoref{thm:C} to  $\widetilde{\EuScript{X}}\to S$ results in the same rational map $\widetilde{\EuScript{X}}\dashrightarrow \cal V\to S$ up to birational equivalence of $\cal V$.  It is harmless to shrink $S$ again to assume that $\widetilde{\EuScript{X}}\to S$ (and its normalised base change) is still flat.

	The inclusion of function fields $K(\cal Y)=K(\widetilde{\cal Y})\supset K(\EuScript{X})\supset K(\widetilde{\cal Y})^p$ determines a unique height one purely inseparable morphism  $\widetilde{\phi}\colon\widetilde{\cal Y}\to\widetilde{\EuScript{X}}$ fitting into the diagram below, and the induced map $\beta\colon\tilde{\cal Y}\to\cal W_{\mathrm{Im}}$ is a morphism since $\tilde{\EuScript{X}}$ resolves the indeterminacies of $\EuScript{X}\dashrightarrow{\cal V}$. Denote by $\widetilde{\cal F}$ the foliation corresponding to $\widetilde{\phi}$.

	\[\xymatrixcolsep{1em}\xymatrix{
		\widetilde{\cal Y} \ar[rd]^-\pi \ar@/_/[rdd]_-\beta \ar[rr]^-{/\widetilde{\cal F}} & & \widetilde{\EuScript{X}} \ar[rd]  \ar[rr] & & \widetilde{\cal Y}^{p} \ar[rdd] \\
		& \cal Y \ar[rr]^-{/\cal F} \ar@{-->}[d]^-{\alpha} & & \EuScript{X} \ar@{-->}[rrd] \ar@{-->}[d] & & \\
		& \cal W_{\Im} \ar[rr] \ar[d] & & \cal V \ar[rr] \ar[d] & & \cal W_{\Im}^{p} \ar[d] \\
		& T \ar[rr] & & S \ar[rr] & & T^{p}}
	\]

	Let $\widetilde{\cal F}'$ denote the saturation of $\widetilde{\cal F}$ in $(g\circ\pi)^*\cal T_{T/S}$.  
	Then define $\widetilde{\cal Q}_{T/S}$ as the quotient.  
	Let $U\subset\cal Y$ be the open set on which $\pi|_{\pi^{-1}(U)}\colon\pi^{-1}(U)\to U$ is an isomorphism.
	Then $\widetilde{\cal F}'|_{\pi^{-1}(U)}\cong\cal F'|_U$, and the maps $\pi^*g^*\cal E\to \det\cal{Q_{\widetilde{F}}}$ and $g^*\cal E\to\det\cal Q_{T/S}$ agree on their restriction to $U$.
	So $\alpha\circ\pi\colon \widetilde{\cal Y}\to\cal W_{\Im}$ agrees with $\beta\colon \widetilde{\cal Y}\to \widetilde{\cal W}_{\Im}$ on $\pi^{-1}(U)$ as maps to $\mathrm{Gr}_T(r,{\cal T}_{T/S})$, and so it follows that these morphisms are the same. 
	
\end{proof}

In view of  \autoref{prop:assume_morphisms} we may assume that $\cal Y\to \cal W\to T$  and $\EuScript{X}\to \cal V\to S$ are morphisms.  We denote by $\frak M_{\cal W/\cal V}$ the movable part of the canonical linear system produced by the normalised base change $\cal Y=(\EuScript{X}\times_{\cal V} \cal W)^{\nu}_{\mathrm{red}}$ over a suitable open subset of $\cal W$, or on the generic fibre of $\cal Y\to \cal W$, as constructed in \autoref{sub:construct_M}.  We  require $\cal Y\to \cal W$ to be a morphism in order to make sense of $\frak M_{\cal W/\cal V}$, but  this is not strictly necessary, as we could define $\frak M_{\cal W/\cal V}$ to be the birational transform of that on a higher model.

\begin{lemma}\label{lem:M_W/V_is_0}
 $\frak M_{\cal W/\cal V}=0$ (over the generic point of $\cal W$).
	\end{lemma}
\begin{proof}
	Denote 
	\[\xymatrix{
	\cal Y\ar_{g_1}[r]\ar@/^1pc/[rr]^-g & \cal W\ar_{g_2}[r] & T
}
	\]
	
	By \autoref{prop:assume_morphisms} we may assume that $\cal Y\to \cal W\to T$ are morphisms.  We are then  free to shrink $\cal V$ and $\cal W$ as much as necessary. In particular we may assume that $\cal T_{\cal W/\cal V}\hookrightarrow g_2^*\cal T_{T/S}$ is saturated and that $\cal Y\to\cal W$ is flat. Then $\cal F$ has the same saturation in $g_1^*\cal T_{\cal W/\cal V}$ and in $g^*\cal T_{T/S}$, so we have the following exact diagram of sheaves on $\cal Y$.
	\[\xymatrix{
		& & &0\ar[d] &\\
0\ar[r] &\cal F'_{\cal W/\cal V}\ar[r] \ar@{=}[d] & g_1^*\cal T_{\cal W/\cal V} \ar[r]\ar[d] & \cal Q_{\cal W/\cal V}\ar[r]\ar[d] &0\\
	0 \ar[r] & \cal F'_{T/S}\ar[r] & g^*\cal T_{T/S}\ar[r] & \cal Q_{T/S} \ar[r]\ar[d] & 0\\
	& & & \widetilde{Q}\ar[d] & \\
	& & &0&
}	
	\]
	
 We know that $\det\cal Q_{T/S}$ is trivial on the generic fibre over $\cal W_{\Im}$, and so from the factorization $\cal Y\to \cal W\to \cal W_{\Im}^{p}$ we know that $(\det\cal Q_{T/S})^{\otimes p}$ is trivial on the generic fibre $Y_{\cal W}$ of $\cal Y\to \cal W$.
But then $\det\cal Q_{T/S}|_Y$ is a globally generated and numerically trivial $\bb Q$-Cartier divisor, and hence trivial as $Y_{\cal W}\to\Spec K(\cal W)$ is projective.

Therefore on a big open subset $U'$ of $Y_{\cal W}$, we have two globally generated vector bundles $\cal Q_{\cal W/\cal V}|_{U'}$ and $\widetilde{\cal Q}|_{U'}$ with 
$\det\cal Q_{\cal W/\cal V}|_{U'}\otimes\det\tilde{\cal Q}|_{U'}\cong\cal O_{U'}$.
Therefore $\det\cal Q_{\cal W/\cal V}|_{U'}\cong\det\widetilde{\cal Q}|_{U'}\cong\cal O_{U'}$.

\end{proof}

\begin{proposition}[\autoref{main_theorem}\autoref{fibre_reduced}]\label{prop:base_change_w_reduced}
 Let $\xi$ be the generic point of $\cal V$.  Then ${\EuScript{X}_{\xi}\otimes_{K(\cal V)}K(\cal W)}$ is reduced.
\end{proposition}

\begin{proof}
	This follows immediately from \autoref{lem:M_W/V_is_0} and \autoref{prop:M=0_iff_reduced}.
\end{proof}

\begin{remark}
 $\EuScript{X}_\xi$ need not be geometrically reduced; see \autoref{exmp:X_K(V)_not_geom_reduced}.
\end{remark}

Next we prove points \autoref{maximal_reduced} and \autoref{pullback_from_w} of \autoref{main_theorem}.

\begin{proposition}[\autoref{main_theorem} \autoref{maximal_reduced}, \autoref{pullback_from_w}]\label{new_univ_prop}
	Let $\cal U\to S$ be a variety with a projective morphism $\alpha\colon \EuScript{X}\to\cal U$
	such that $\EuScript{X}\times_\cal U \Spec (K((\cal U\times_S T)_{\red}))$ is reduced, where $K((\cal U\times_S T)_{\red})$ is the function field of the variety $(\cal U\times_S T)_{\red}$. 
	Then there is a  rational map $(\cal U\times_S T)_{\red}\dashrightarrow\cal W_{\mathrm{Im}}$ factoring the natural map from $\cal Y$. 
	
	Furthermore, if $\alpha_*\cal O_{\EuScript{X}}=\cal O_\cal U$, then there is a rational map $\cal U\dashrightarrow\cal V$ factoring $\EuScript{X}\dashrightarrow\cal V$.
	
\end{proposition}

\begin{proof}
Since we only require rational maps, we will shrink $\cal U$ as needed throughout the proof.
We have a diagram 

\[\xymatrix{
	\cal Z\ar[r]^-{i'}\ar[rd] \ar@/_2pc/[rrdd]_-h &\EuScript{X}\times_\cal U  ((\cal U\times_S T)_{\red}) \ar[r]^-i \ar[d]^-{\alpha_r} & \EuScript{X}\times_S T \ar[d]^-{\alpha_T} \ar[r] & \EuScript{X} \ar[d]^-\alpha \\
	&(\cal U\times_S T)_{\red} \ar[r]^-{i_\cal U}\ar[rd]^-{h_\cal U} & \cal U\times_S T \ar[r]\ar[d] & \cal U\ar[d] \\
& & T\ar[r] & S.}
\]

As $i_{\cal U}$ is a closed immersion, so is $i$, but then by assumption $\EuScript{X}\times_\cal U  ((\cal U\times_S T)_{\red}) $ is reduced, and so $i'$ is an isomorphism.

The conormal sequence on $\cal U$ gives 
\[
\xymatrix{
I_\cal U/I_{\cal U}^2 \ar[r] & h_\cal U^*\Omega_{T/S} \ar[r] & \Omega_{(\cal U\times_S T)_{\red}/\cal U} \ar[r] & 0}.
\]

But we have $\alpha_r^*\Omega_{(\cal U\times_S T)_{\red}/\cal U} \cong\Omega_{\cal Z/\EuScript X}$, and so the pullback of the right-hand map is exactly 
\[\xymatrix{
h^*\Omega_{T/S}\ar[r] & \Omega_{\cal Z/\EuScript X} \ar[r] & 0},
\]
whose top wedge defines the map $\cal Y\dashrightarrow\cal W_{\Im}$ on the open set on which $\cal Z\cong\cal Y$ and $\cal Y$ is smooth, by \autoref{lem:plucker}.
Therefore $\alpha_r$ factors into the map $\cal Z\dashrightarrow \cal W$.

Similarly, taking top wedge of this map commutes with pullback, so we obtain equality of linear systems $\alpha_r^*\frak M_{\cal U}=\frak M$.

This gives a factorisation $\EuScript{X} \dashrightarrow (\cal U\times_S T)_{\red}^p\dashrightarrow\cal W_{\mathrm{Im}}^p$, so we get a rational map between the normalisations of $(\cal U\times_S T)_{\red}^p$ and $\cal W_{\mathrm{Im}}^p$ in $K(\EuScript{X})$. The latter is $\cal V$ by definition, so if moreover $\alpha_*\cal O_{\EuScript{X}}=\cal O_\cal U$ we obtain $\EuScript{X}\to\cal U\dashrightarrow\cal V$.
\end{proof}

\section{Fano varieties}\label{sec:fano_section}

In this section we prove some consequences of our main results to Fano varieties over imperfect fields. We first prove \autoref{thm:MFS}, whose statement we recall below:

\begin{theorem}[\autoref{thm:MFS}]
	Let $K$ be the function field of a variety over a perfect field of characteristic $p>0$. Let $X$ be a variety over $K$, and assume the following:
	\begin{itemise}
		\item $X$ is normal, $\mathbb{Q}$-factorial and $\rho(X)=1$,
		\item $X$ is geometrically irreducible over $K$,
		\item there is an effective divisor $B$ such that $K_X+B\equiv 0$, and
		\item if $n=\dim(X)$ then $p>2n+1$.
	\end{itemise}
	
	Then there is a birational morphism $\phi\colon\tilde{X}\to X$ and a contraction $\widetilde{X}\to V$ of relative dimension at least $1$, such that $(\widetilde{X}\otimes_K{K^{1/p}})_{\red}$ is birational to $\widetilde{X}\times_V(V\otimes_K{K}^{1/p})_{\red}$.
\end{theorem}

\begin{proof}[Proof of \autoref{thm:MFS}]
Let $Y^{1/p}=(X\otimes_KK^{1/p})_{\red}^{\nu}$ and $\phi_1$ be the map $Y^{1/p}\to X$.  By \cite[Proposition 2.4]{tanaka_behavior_2016}, $\rho(Y^{1/p})=\rho(X)=1$, and by \cite[Lemma 2.5]{tanaka_behavior_2016} $Y^{1/p}$ is $\mathbb{Q}$-factorial.
	By \autoref{main_theorem} we obtain a formula 
	\[\phi_1^*(K_X+B)\sim K_{Y^{1/p}}+(p-1)(\frak{F}+\frak{M})+\phi_1^*B\equiv0
	\]  
	where $\mathrm{Bs}(\frak{M})$ has codimension at most $2$.  In our situation $\frak{F}$, $\frak{M}$  are effective Weil divisors and $\phi_1^*B$ an effective $\mathbb{Q}$-divisor on a $\mathbb{Q}$-factorial variety of Picard rank $1$, so they are each nef $\mathbb{Q}$-Cartier divisors. 
	Let $g\colon Y^{1/p}\dashrightarrow W$ be the rational map from \autoref{main_theorem}, which matches that given by the linear system $\frak{M}$ up to a finite morphism.  Denote the rational map induced by $\frak M$ as $g_{\frak M}$. 
	
	Suppose that $g_{\frak M}$ is generically finite, and we aim to find a contradiction.   The key is that if we fix a point $P$ of $Y^{1/p}$ where $g_{\frak M}$ is defined as a morphism and is finite, and a curve $\Gamma$ through $P$, this assumption allows us to choose a representative of $\frak M$ which passes through $P$ but does not contain $\Gamma$.
	
Let $Y^{\perf}=(Y^{1/p}\otimes_{K^{1/p}}K^{1/p^{\infty}})^{\nu}_{\red}$ and let $\phi_\infty$ be the map $Y^{\perf}\to Y^{1/p}$.  By \cite[Theorem 1.1]{pw} we can fix  an effective Weil divisor $C$ on $Y^{\perf}$ such that
\[\phi_{\infty}^*\phi_1^*(K_X+B)=K_{Y^{\perf}}+(p-1)(\phi_{\infty}^*(\frak F+\frak M)+C)+\phi_{\infty}^*\phi_1^*B.
\]
	As before $Y^{\perf}$ is $\mathbb{Q}$-factorial and $\rho(Y^{\perf})=1$, and so $C$ is nef.
	
	Finally let $Y=Y^{\perf}\otimes_{K^{1/p^\infty}}\overline{K}$, which is an irreducible normal variety  because $X$ was geometrically irreducible, and let $\overline{\phi}$ be the map $Y\to Y^{\perf}$.   We have $K_{Y}=\overline{\phi}^*K_{Y^{\perf}}$, and so 
	\begin{gather*}\overline{\phi}^*(K_{Y^{\perf}}+(p-1)(\phi_{\infty}^*(\frak{F}+\frak{M})+C)+\overline{\phi}^*\phi_\infty^*\phi_{1}^*B) \\
	=K_Y+(p-1)((\frak{F}_Y+\frak{M}_Y)+C_Y)+B_Y\equiv 0
	\end{gather*}
where $\frak{F}_Y$, $\frak{M}_Y$, $B_Y$ and $C_Y$ are nef $\mathbb{Q}$-Cartier Weil divisors obtained by pulling back the corresponding divisors from $Y^{\perf}$.  
	
	Cut $Y$ by general hyperplanes to obtain a smooth curve $D$ contained within the smooth locus of $Y$.  Fix a point $P$ of $Y$ which lies on $D$ but is not contained in the support of $\frak{F}_Y$, $B_Y$ or $C_Y$ (we can do that by generality of the hyperplanes).  Furthermore, we can choose $P$ to be over the locus on which the map $g_{\frak M}$ is defined as a morphism and is finite.  
	
	Now apply bend and break \cite[II.5.8]{kollar_rational_1996} to the nef $\mathbb{Q}$-Cartier divisor $-K_Y$ and the curve $D$ to obtain a curve $\Gamma$ through the point $P$ with $-K_Y\cdot\Gamma\leq 2n$.  Then 
	\[(p-1)\frak{M}_Y\cdot\Gamma\leq ((p-1)(\frak{F}_Y+\frak{M}_Y+C_Y)+B_Y)\cdot\Gamma=-K_Y\cdot \Gamma\leq 2n.
	\]
	
	But $P$ is chosen in such a way that we could have chosen an element $M_P\in \frak{M}$ through $P$ which does not contain the curve $\Gamma$.
	This implies that $\frak{M}_Y\cdot \Gamma$ is greater than a positive integer, and in particular
	 $\frak{M}_Y\cdot\Gamma\geq 1$.  The inequality from bend and break now gives $p-1\leq 2n$, that is $p\leq 2n+1$.
	
	It follows that if $p>2n+1$, the rational map $Y^{1/p}\dashrightarrow W$ is not birational, and \autoref{main_theorem} gives the required property.
	
\end{proof}

\begin{remark}
Note that \autoref{thm:MFS} applies equally to the pair obtained by base change to $K^{1/p}$, that is $(Y^{1/p},(p-1)(\frak{F}+\frak{M})+\phi_{1}^*B)$.   So in fact we inductively obtain a sequence of contractions as in \autoref{thm:MFS} at each Frobenius base change as we continue towards $K^{1/p^{\infty}}$.
\end{remark}

In the next corollaries we may reduce to the case where the ground field is a function field using the process outlined in \autoref{arbitrary_fields}, and so apply \autoref{main_theorem}.

\begin{proof}[Proof of \autoref{cor:dp_hirzebruch}]
As we assume $X$ is geometrically non-reduced, we deduce from \autoref{main_theorem} and \autoref{arbitrary_fields} that the linear system $|C|$ has nonzero movable part. But the exceptional section $D$ on the Hirzebruch surface is fixed so this cannot occur.
\end{proof}

\begin{proof}[Proof of \autoref{regular_dp}]
If $p=3$ then $(Y,C)=(\bb P^2,L)$ or $(S_d,F)$. In both of these situations, any linear system of $\dim\geq 1$ inside $|C|$ has no fixed part.

\end{proof}

\begin{proof}[Proof of \autoref{cor:dp_product}]
Since we assume that $(Y,(p-1)C)$ is of the type $(\bb P^1\times\bb P^1,F)$ (see \cite[Theorem 4.1]{pw}), $C$ is a prime divisor and
$X\otimes_K K^{1/p}$ is non-reduced, and we have $Y\cong (X\otimes_K K^{1/p})_{\red}\otimes_{K^{1/p}} \overline{K}$. That is, $(X\otimes_K K^{1/p})_{\mathrm{red}}$ is geometrically normal.
The linear system $\frak M_{K^{1/p}/K}$ is 1-dimensional so $V$ is a curve, and by universal properties $X\times_V K((V\otimes_K K^{1/p})_{\red})$ is isomorphic to $(X\otimes_K K^{1/p})_{\red}$.
So over the generic point of $V$
\[ Y \cong (X\times_V(V\otimes_K K^{1/p})_{\red}^\nu)\otimes_{K^{1/p}} \overline{K} \cong
X\times_V(V\otimes_K \overline{K})_{\red}^\nu
\]
where the second isomorphism holds because $X\times_V K((V\otimes_K K^{1/p})_{\red})$ is geometrically normal over the generic point of $V$.
\end{proof}

\section{Examples}\label{sec:examples}

In this section $k$ will always denote a perfect field of characteristic $p>0$.

Our first example is the archetypal example of a geometrically non-reduced variety.  We compute our linear system explicitly in this case.

\begin{example}\label{exmp:pm_fermat_example}
Let $S=\bb A^{n+1}_k$ with coordinates $s_0,\ldots,s_n$, and let
\[\EuScript{X}=\left(\left(\sum_{i=0}^{n} s_i x_i^{p^m}\right)+x_{n+1}^{p^m}=0\right)\subset\bb P^{n+1}_S.
\]
Let $\tau\colon T\to S$ be relative Frobenius over $k$, so that $\Omega_{T/S}=\Omega_{T/k}$.
Then $(\EuScript{X}\times_S T)_{\red}$ is normal, so $\frak F_{T/S}=0$ and
\[\cal Y=\left(\left(\sum_{i=0}^{n} s_i^{1/p} x_i^{p^{m-1}}\right)+x_{n+1}^{p^{m-1}}=0\right)\subset\bb P^{n+1}_T.
\]

We compute $\frak M_{T/S}$ explicitly on the affine cover $\{(x_j\neq 0) \mid 0\leq j\leq n\}$ of $\cal Y$.
We have
\[ \Omega_{\cal Y/\cal X}(x_j\neq 0)=\Omega_{k[\{s_l^{1/p} \mid l\neq j\},x_0,\ldots,x_{n+1}] / k[\{s_l \mid l\neq j\},x_0,\ldots,x_{n+1}]}
=\langle d_{\cal Y/\cal X} s_l^{1/p} \mid l\neq j\rangle
\]
because $d_{\cal Y/\cal X}s_j^{1/p} = -\sum_{l\neq j} (\tfrac{x_l}{x_j})^{p^{m-1}} d_{\cal Y/\cal X} s_l^{1/p}$ on $(x_j\neq 0)$.

The differentials $\{ d_{T/S} s_i^{1/p} \mid 0\leq i\leq n\}$ form a basis of $\Omega_{T/S}$,
and for each $0\leq i\leq n$, the image of $d_{T/S}s_i^{1/p}\otimes 1$ in $H^0(\mathcal{Y},\Omega_{\mathcal{Y}/\mathcal{X}})$ in terms of the bases on the affine charts described above,  is given by
\[\begin{cases}
d_{\cal Y/\cal X} s_i^{1/p} & \text{on }(x_j\neq 0) \text{ for all }j\neq i, \\
-\sum_{l\neq i}(\tfrac{x_l}{x_i})^{p^{m-1}} d_{\cal Y/\cal X}s_l^{1/p} & \text{on }(x_i\neq 0).
\end{cases}\]
So
\[ \bigwedge_{k\neq i} d_{T/S}s_k^{1/p} \mapsto \begin{cases}
-(\tfrac{x_i}{x_j})^{p^{m-1}} \bigwedge_{l\neq j} d_{\cal Y/\cal X}s_l^{1/p} & \text{on }(x_j\neq 0) \text{ for all }j\neq i, \\
\bigwedge_{k\neq i} d_{\cal Y/\cal X}s_k^{1/p} & \text{on }(x_i\neq 0).
\end{cases}
\]
and therefore
$\det\Omega_{\cal Y/\cal X}\cong\cal O_{\bb P^n_T}(p^{m-1})|_\cal Y$ and
\[ \frak M_{T/S} = \langle x_i^{p^{m-1}} \mid 0\leq i\leq n \rangle \subset \left| \cal O_{\bb P^2_T}(p^{m-1})|_\cal Y\right|
\]
and defines a map to $\Proj k[s_0^{1/p},\ldots,s_n^{1/p}][x_0^{p^{m-1}},\ldots,x_n^{p^{m-1}}]$.

So the factorisation from \autoref{main_theorem} is trivial; that is, $\EuScript{X}=\cal V$ and $\cal Y=\cal W$.

In particular if $m=1$, then for any $n$ and $p<n+2$ the generic fibre $X$ is an $n$-dimensional regular, geometrically non-reduced Fano variety over $k(s_0,\ldots,s_n)$.
\end{example}

The following is an example which illustrates that the generic fibre of $\cal X\to \cal V$ need not be geometrically reduced.

\begin{example}\label{exmp:X_K(V)_not_geom_reduced}
Let $S=\bb A^2_{k,(s,t)}$ and
\[ \EuScript{X}=(sx^{p^m}+ty^{p^m}+z^{p^m}=xu^{p^n}+yv^{p^n}+zw^{p^n}=0)\subset\bb P^2_{S,[x:y:z]}\times_S\bb P^2_{S,[u:v:w]}.
\]
If $\tau\colon T\to S$ is relative Frobenius over $k$ then $\Omega_{T/S}=\Omega_{T/k}$ and the normalised base change is
\[\cal Y=(\EuScript{X}\times_S T)_{\red}=(s^{1/p}x^{p^{m-1}}+t^{1/p}y^{p^{m-1}}+z^{p^{m-1}}=xu^{p^n}+yv^{p^n}+zw^{p^n}=0)
\]
in $\bb P^2_{T,[x:y:z]}\times_T\bb P^2_{T,[u:v:w]}$.
So $\frak F_{T/S}=0$ and the same computation as in \autoref{exmp:pm_fermat_example} shows that $\frak M_{T/S}=
\langle x^{p^{m-1}},y^{p^{m-1}}
\rangle
\subsetneq
|\cal O_{\bb P^2_T\times_T\bb P^2_T}(p^{m-1},0)|_\cal Y|$. Therefore the factorisation of \autoref{main_theorem} is
\[\xymatrix{
\cal Y \ar[r] \ar[d] & \EuScript{X} \ar[d] \\
\cal W=\Proj \frac{k[s^{1/p},t^{1/p}][x,y,z]}{(s^{1/p}x^{p^{m-1}}+t^{1/p}y^{p^{m-1}}+z^{p^{m-1}})} \ar[r] \ar[d] & \cal V = \Proj \frac{k[s,t][x,y,z]}{(sx^{p^m}+ty^{p^m}+z^{p^m})} \ar[d] \\
T \ar[r] & S}.
\]

$\EuScript{X}\times_\cal V\cal W$ is reduced, but the generic fibre $\EuScript{X}_\xi$ of $\EuScript{X}\to\cal V$ is not geometrically reduced over $K(\cal V)$; indeed
$\EuScript{X}_\xi$ is given by $(xu^p+yv^p+zw^p=0)\subset\bb P^2_{K(\cal V)}$ and becomes non-reduced upon base change to $K(\cal V)^{1/p}$.
\end{example}

By taking products of $p^m$-Fermat hypersurfaces, we obtain examples where both $\frak F$ and $\frak M$ are nonzero.

\begin{example}
Let $S=\bb A^3_{k,(r,s,t)}$ and
\[\EuScript{X}=(sx^{p^m}+ty^{p^m}+z^{p^m}= ru^{p^n}+sv^{p^n}+w^{p^n}=0)\subset\bb P^2_{S,[x:y:z]}\times_S\bb P^2_{S,[u:v:w]}.
\]
If $T=\Spec k[r,s^{1/p},t^{1/p}]\to S$ then $\Omega_{T/S}$ is freely generated by $d_{T/S} s^{1/p}$ and $d_{T/S} t^{1/p}\rangle $, and the normalised base change is
\[\cal Y=(s^{1/p}x^{p^{m-1}}+t^{1/p}y^{p^{m-1}}+z^{p^{m-1}}=r^{1/p}u^{p^{n-1}}+s^{1/p}v^{p^{n-1}}+w^{p^{n-1}}=0)
\]
inside $\bb P^2_{S^{p},[x:y:z]}\times_{S^{p}}\bb P^2_{S^{(1)},[u:v:w]}$.
One can compute on the affine open subsets $\{(x\neq 0)\cap (u\neq 0), (y\neq 0)\cap (u\neq 0), (y\neq 0)\cap (v\neq 0)\}$ of $\cal Y$ that
\begin{gather*}
\frak F_{T/S}=u^{p^{n-1}} \in \left| \cal O_{\bb P^2\times\bb P^2}(0,p^{n-1})|_\cal Y \right|, \\
\frak M_{T/S}=\langle x^{p^{m-1}},y^{p^{m-1}}\rangle \subsetneq \left| \cal O_{\bb P^2\times\bb P^2}(p^{m-1},0)|_\cal Y \right|.
\end{gather*}
Indeed, by adjunction $(p-1)(\frak M_{T/S}+\frak F_{T/S})=(p-1)\frak C_{T/S}$ is in the linear equivalence class of $\phi^*K_\EuScript{X}-K_\cal Y=\cal O_{\bb P^2\times\bb P^2}(p^m-p^{m-1},p^n-p^{n-1})|_\cal Y$.

$\frak M_{T/S}$ defines a map to $\Proj k[r,s^{1/p},t^{1/p}][x^{p^{m-1}},y^{p^{m-1}}]$ and so the factorisation of \autoref{main_theorem} is
\[\xymatrix{
\cal Y \ar[r] \ar[d] & \EuScript{X} \ar[d] \\
\cal W=\Proj \frac{k[r,s^{1/p},t^{1/p}][x,y,z]}{(s^{1/p}x^{p^{m-1}}+t^{1/p}y^{p^{m-1}}+z^{p^{m-1}})} \ar[r] \ar[d] & \cal V = \Proj \frac{k[r,s,t][x,y,z]}{(sx^{p^m}+ty^{p^m}+z^{p^m})} \ar[d] \\
T \ar[r] & S.}
\]

\end{example}

\bibliographystyle{acm}
\bibliography{references_nonred}


\end{document}